\documentclass[reqno]{amsart}
\usepackage{setspace,tikz,xcolor,mathrsfs,listings,multicol}
\usepackage{amssymb}
\usepackage{rotating}
\usepackage[centertableaux]{ytableau}
\usepackage{enumerate}
\usepackage[all,cmtip]{xy}
\usetikzlibrary{arrows,matrix,positioning}
\tikzset{>=stealth',
  head/.style = {fill = white, text=black},
  plaque/.style = {draw, rectangle, minimum size = 10mm}, 
  pil/.style={->,thick},
  junct/.style = {draw,circle,inner sep=0.5pt,outer sep=0pt, fill=black}
  }

\usepackage[colorlinks=true, pdfstartview=FitV, linkcolor=blue, citecolor=blue, urlcolor=blue]{hyperref}

\newcommand{\G}{\mathfrak{G}}  
\newcommand{\wG}{\mathfrak{J}}  
\newcommand{\hG}{\mathfrak{H}}  
\newcommand{\dG}{\mathfrak{g}}  
\newcommand{\dwG}{\mathfrak{j}}  

\newcommand{\xx}{\mathbf{x}}
\newcommand{\HH}{\mathcal{H}}

\newcommand{\iso}{\cong}
\newcommand{\bbb}{\mathsf{b}}
\newcommand{\gp}{\mathsf{g}}

\newcommand{\abs}[1]{\lvert #1 \rvert}

\newcommand{\fsl}{\mathfrak{sl}}
\newcommand{\gl}{\mathfrak{gl}}
\newcommand{\sym}{S}
\newcommand{\ml}[1]{ \mathbf{\color{darkred} #1} } 

\DeclareMathOperator{\wt}{wt} 
\DeclareMathOperator{\rd}{rd} 
\DeclareMathOperator{\ssyt}{SST} 
\DeclareMathOperator{\svt}{SVT} 
\DeclareMathOperator{\mvt}{MVT} 
\DeclareMathOperator{\hvt}{HVT} 
\DeclareMathOperator{\rpp}{RPP} 
\DeclareMathOperator{\vst}{VST} 
\DeclareMathOperator{\Gr}{Gr} 
\DeclareMathOperator{\RSK}{RSK} 

\newcommand{\ZZ}{\mathbb{Z}}

\newcommand{\CC}{\mathbb{C}}

\newcommand{\mcB}{\mathcal{B}}

\newcommand{\mcF}{\mathcal{F}}

\newcommand{\mcP}{\mathcal{P}}

\newcommand{\bplus}{{\color{blue}+}}
\newcommand{\bminus}{{\color{red}-}}

\usepackage{listings}
\lstdefinelanguage{Sage}[]{Python}
{morekeywords={False,sage,True},sensitive=true}
\definecolor{dblackcolor}{rgb}{0.0,0.0,0.0}
\definecolor{dbluecolor}{rgb}{0.01,0.02,0.7}
\definecolor{dgreencolor}{rgb}{0.2,0.4,0.0}
\definecolor{dgraycolor}{rgb}{0.30,0.3,0.30}

\definecolor{darkred}{rgb}{0.7,0,0} 
\newcommand{\defn}[1]{{\color{darkred}\emph{#1}}} 

\usepackage[colorinlistoftodos]{todonotes}

\theoremstyle{plain}
\newtheorem{thm}{Theorem}[section]
\newtheorem{lemma}[thm]{Lemma}

\newtheorem{prop}[thm]{Proposition}
\newtheorem{cor}[thm]{Corollary}
\theoremstyle{definition}
\newtheorem{dfn}[thm]{Definition}
\newtheorem{ex}[thm]{Example}
\newtheorem{remark}[thm]{Remark}

\numberwithin{equation}{section}

\begin{document}
\title[Crystal structures for canonical Grothendiecks]{Crystal structures for canonical Grothendieck functions}

\author[G.~Hawkes]{Graham Hawkes}
\address[G.~Hawkes]{Department of Mathematics, University of California, Davis, One Shields Avenue, Davis, CA 95616, USA}
\email{hawkes@math.ucdavis.edu}
\urladdr{https://www.math.ucdavis.edu/~hawkes/}

\author[T.~Scrimshaw]{Travis Scrimshaw}
\address[T.~Scrimshaw]{School of Mathematics and Physics, The University of Queensland, St.\ Lucia, QLD 4072, Australia}
\email{tcscrims@gmail.com}
\urladdr{https://people.smp.uq.edu.au/TravisScrimshaw/}


\thanks{TS was partially supported by the Australian Research Council DP170102648.}

\begin{abstract}
We give a $U_q(\mathfrak{sl}_n)$-crystal structure on multiset-valued tableaux, hook-valued tableaux, and valued-set tableaux, whose generating functions are the weak symmetric, canonical, and dual weak symmetric Grothendieck functions, respectively.
We show the result is isomorphic to a (generally infinite) direct sum of highest weight crystals, and for multiset-valued tableaux and valued-set tableaux, we provide an explicit bijection.
As a consequence, these generating functions are Schur positive; in particular, the canonical Grothendieck functions, which was not previously known.
We also give an extension of Hecke insertion to express a dual stable Grothendieck function as a sum of Schur functions.
\end{abstract}

\maketitle

\section{Introduction}
\label{sec:introduction}

The Grassmannian $\Gr(n, k)$ is the set of $k$-dimensional hyperplanes in $\CC^n$. Lascoux and Sch\"utzenberger~\cite{LS82,LS83} introduced Grothendieck polynomials to represent the K-theory ring of the Grassmannian. In particular, they correspond to the K-theory classes of structure sheaves of Schubert varieties, and so they are indexed by permutations in $\sym_n$. By taking the stable limit of $n \to \infty$, Fomin and Kirillov~\cite{FK94,FK96} initiated the study of stable Grothendieck functions, where they also replaced the sign corresponding to the degree by a parameter $\beta$ (which corresponds to taking the connective K-theory~\cite{Hudson}).
Stable Grothendieck functions have been well-studied using a variety of methods; see for example~\cite{BKSTY08,BM12,BS16,Buch02,GMPPRST16,IIM17,IN09,IS14,Iwao19,LamPyl07,Lenart00,Monical16,MS13,MS14,MPS18,MPS18II,PP16,PS18,PS19,PY:genomic,PY18,TY09,WZJ16} and references therein.

The subset of stable Grothendieck functions corresponding to Grassmannian permutations are called symmetric Grothendieck functions and form a basis for (an appropriate completion of) the ring of symmetric functions over $\ZZ[\beta]$.
Recall that Schur functions correspond to the characters of the general-linear Lie algebra $\gl_n$ when restricted to $n$ variables.
Symmetric Grothendieck functions $\G_{\lambda}$ are known to be Schur positive~\cite{Lenart00} with a finite expansion in each degree $\beta$, and so we can apply the involution $\omega$ that sends a Schur function $s_{\mu}$ to the Schur function $s_{\mu'}$ of the conjugate $\mu'$ of $\mu$.
The resulting basis is known as the weak stable Grothendieck functions $\wG_{\lambda}$.
Since the basis of Grothendieck functions is a (upper) filtered basis, we can consider its Hopf dual basis under the Hall inner product, which can be defined by considering Schur functions as an orthonormal basis, called the dual symmetric Grothendieck functions and denoted by $\dG_{\lambda}$.
By also applying $\omega$, we obtain the dual weak symmetric Grothendieck functions $\dwG_{\lambda}$.
Furthermore, each of the above families are known to have combinatorial interpretations:
\begin{itemize}
\item symmetric Grothendieck functions using set-valued tableaux,
\item weak symmetric Grothendieck functions using multiset-valued tableaux,
\item dual symmetric Grothendieck functions using reverse plane partitions, and
\item dual weak symmetric Grothendieck polynomials using valued-set tableaux.
\end{itemize}

In an effort to unify the bases $\{\G_{\lambda}\}_{\lambda}$ and $\{\wG_{\lambda}\}_{\lambda}$ by constructing a basis invariant under $\omega$, Yeliussizov introduced in~\cite{Yel17} the canonical Grothendieck functions $\hG_{\lambda}$ and fused the corresponding combinatorics in hook-valued tableaux. Furthermore, up to a coefficient of $(\alpha + \beta)$, the canonical Grothendieck functions have the same structure coefficients and coproduct as the symmetric Grothendeick functions. Similarly, he defined the dual canonical Grothendieck functions as the corresponding dual basis and described it combinatorially using rim border tableaux and showed they are Schur positive.

Since a Schur function restricted to $n$ variables is a character of the special-linear Lie algebra $\fsl_n$, a generating function of some set $\mathcal{B}$ that is Schur positive implies that there should be a $U_q(\fsl_n)$-crystal structure~\cite{K90,K91} on $\mathcal{B}$ with each connected component isomorphic to the highest weight crystal $B(\lambda)$ for every $s_{\lambda}$ summand.
Indeed, this was done for symmetric Grothendieck functions~\cite{MPS18} and for dual symmetric Grothendieck functions~\cite{Galashin17}.
Thus, a natural question is to construct such crystals on multiset-valued tableaux, hook-valued tableaux, valued-set tableaux, and rim border tableaux.
In this paper, we construct such a $U_q(\fsl_n)$-crystal structure on the first three combinatorial objects: multiset-valued tableaux, hook-valued tableaux, and valued-set tableaux.

Furthermore, we show that multiset-valued tableaux have many analogous results from~\cite{MPS18} for set-valued tableaux. More specifically, we extend the notion of the uncrowding crystal isomorphism from~\cite[Sec.~6]{Buch02} to an explicit crystal isomorphism from mutliset-valued tableaux to the usual crystal on semistandard tableaux. Furthermore, we extend Hecke insertion~\cite{BKSTY08} to give a crystal structure on weakly decreasing factorizations and give a positive Schur expansion of general weak stable Grothendieck functions. We also have chosen our reading word on multiset-valued tableaux so that it is a crystal embedding.

Conversely, whenever we have a crystal structure on a set $\mathcal{B}$, the corresponding generating function is Schur positive. Our final result is constructing a $U_q(\fsl_n)$-crystal structure on hook-valued tableaux, which immediately implies that the canonical Grothendieck functions are Schur positive. It was not previously known that the canonical Grothendieck functions are Schur positive.

Our crystal structure on hook-valued tableaux is a combination of the crystal structures on set-valued tableaux and multiset-valued tableaux. However, we are not able to provide an explicit isomorphism with a highest weight crystal and instead must rely on the Stembridge axioms~\cite{Stembridge03}. Indeed, the set-valued (resp.\ multset-valued) tableaux crystal structure preserves rows (resp.\ columns), each of which is isomorphic to hook shape, and so the crystal structures are incompatible with no straightforward extension of uncrowding.

In addition to the crystal structure on valued-set tableaux (which we describe using two different reading words), we provide an analog of uncrowding that we call inflation. That is we give an explicit crystal isomorphism from valued-set tableaux to the usual crystal on semistandard tableaux that is based on~\cite[Thm.~9.8]{LamPyl07}. As with uncrowding, inflation is based on the Robinson--Schensted--Knuth (RSK) algorithm being a crystal isomorphism and recording the difference between the shapes.

Since dual canonical Grothendieck are Schur positive~\cite[Thm.~9.8]{Yel17}, there should exist a $U_q(\fsl_n)$-crystal structure on rim border tableaux with an additional marking of all interior boxes by either $\alpha$ or $\beta$ as the exponent of $(\alpha + \beta)$ corresponds to the number of interior boxes. However, the crystal structure appears to be more complicated than simply combining the crystal structures on reverse plane partitions and valued-set tableaux. Thus, it remains an open problem to construct a $U_q(\fsl_n)$-crystal on marked rim border tableaux. Moreover, it is unlikely that inflation will extend to marked rim border tableaux.

This paper is organized as follows.
In Section~\ref{sec:background}, we provide the necessary background.
In Section~\ref{sec:crystal_MVT}, we give a crystal structure on multiset-valued tableaux, the uncrowding map, and a variation of Hecke insertion for weak stable Grothendieck polynomials.
In Section~\ref{sec:hooks}, we construct a crystal structure on hook-valued tableaux.
In Section~\ref{sec:vst}, we construct a crystal structure on valued-set tableaux and give the inflation map.

\subsection*{Acknowledgements}
The authors would like to thank Rebecca Patrias, Oliver Pechenik, and Damir Yeliussizov for useful discussions and Anne Schilling for comments on an earlier version of this manuscript.
This work benefited from computations using \textsc{SageMath}~\cite{sage, combinat}.

\section{Background}
\label{sec:background}

In this section, we give the necessary background on the crystal structure on set-valued tableaux and on (weak) symmetric/canonical Grothendieck functions.
We use English convention for partitions and tableaux.
Let $\xx = (x_1, x_2, x_3, \ldots)$ be a countable sequence of indeterminants.
Let $\fsl_n$ denote the special linear Lie algebra (\textit{i.e.}, the simple Lie algebra of type $A_{n-1}$) over $\CC$ and $U_q(\fsl_n)$ the corresponding Drinfel'd--Jimbo quantum group.
Let $\lambda = (\lambda_1, \lambda_2, \dotsc, \lambda_{\ell})$ be a partition; a sequence of weakly decreasing positive integers.
Let $\ell(\lambda) = \ell$ denote the length of $\lambda$.

\subsection{Semistandard tableaux, set-valued tableaux, and crystals}

A \defn{(semistandard) set-valued tableau of shape $\lambda$} is a filling $T$ of the boxes of Young diagram of $\lambda$ by finite nonempty sets of positive integers so that rows are weakly increasing and columns are strictly increasing in the following sense:
For every set $A$ to the left of a set $B$ in the same row, we have $\max A \leq \min B$, and for $C$ below $A$ in the same column, we have $\max A < \min C$.
A set-valued tableau is a \defn{semistandard (Young) tableau} if all sets have size $1$.
Let $\svt^n(\lambda)$ (resp.~$\ssyt^n(\lambda)$) denote the set of set-valued (resp.\ semistandard) tableaux of shape $\lambda$ with entries at most $n$.

In~\cite{MPS18}, a $U_q(\fsl_n)$-crystal structure, in the sense of Kashiwara~\cite{K90,K91}, was given on $\svt^n(\lambda)$.
Recalling this crystal structure, we begin with the \defn{crystal operators} $e_i, f_i \colon \svt^n(\lambda) \to \svt^n(\lambda) \sqcup \{0\}$, where $i \in I := \{1, \dotsc, n-1\}$.

\begin{dfn}
\label{defn:svt_crystal_ops}
Fix some $T \in \svt^n(\lambda)$ and $i \in I$.
Write $\bplus$ above each column of $T$ containing $i$ but not $i+1$, and write $\bminus$ above each column containing $i+1$ but not $i$. Next cancel signs in ordered pairs $\bminus \bplus$ until obtaining a sequence of the form $\bplus \cdots \bplus \bminus \cdots \bminus$ called the \defn{$i$-signature}.
\begin{description}
\item[\defn{$e_i T$}] If there is not a $\bminus$ in the resulting sequence, then $e_i T = 0$. Otherwise let $\bbb$ correspond to the box of the leftmost uncanceled $\bminus$. Then $e_i T$ is given by one of the following:
\begin{itemize}
\item if there exists a box $\bbb^{\leftarrow}$ immediately to the left of $\bbb$ that contains an $i+1$, then remove the $i+1$ from $\bbb^{\leftarrow}$ and add an $i$ to $\bbb$;
\item otherwise replace the $i+1$ in $\bbb$ with an $i$.
\end{itemize}

\item[\defn{$f_i T$}] If there is not a $\bplus$ in the resulting sequence, then $f_i T = 0$. Otherwise let $\bbb$ correspond to the box of the rightmost uncanceled $\bplus$. Then $f_i T$ is given by one of the following:
\begin{itemize}
\item if there exists a box $\bbb^{\rightarrow}$ immediately to the right of $\bbb$ that contains an $i$, then remove the $i$ from $\bbb^{\rightarrow}$ and add an $i+1$ to $\bbb$;
\item otherwise replace the $i$ in $\bbb$ with an $i+1$.
\end{itemize}
\end{description}
\end{dfn}

For a set-valued tableau $T \in \svt^n(\lambda)$ the \defn{weight} is defined as
\begin{equation}
\label{eq:weight_defn}
\wt(T) := x_1^{m_1} x_2^{m_2} \cdots x_n^{m_n},
\end{equation}
where $m_i$ is the number of occurrences of $i$ in $T$.
Denote $\abs{T} := \sum_{i=1}^n m_i$.
Define the statistics
\[
\varepsilon_i(T) = \max \{k \in \ZZ_{\geq 0} \mid e_i^k T \neq 0\},
\qquad\quad
\varphi_i(T) = \max \{k \in \ZZ_{\geq 0} \mid f_i^k T \neq 0\}.
\]
This gives a $U_q(\fsl_n)$-crystal structure on $\svt^n(\lambda)$.\footnote{The standard references for crystal structures consider the weight as an additive group, but we consider it as a multiplicative group because it useful for defining polynomials in the sequel.}
In particular, for $T, T' \in \svt^n(\lambda)$, we have
\[
e_i T = T' \Longleftrightarrow T = f_i T'.
\]
We say a $T \in \svt^n(\lambda)$ is \defn{highest weight} if $e_i T = 0$ for all $i \in i$.
For more details on crystals, we refer the reader to~\cite{BS17,K91}.

When we restrict this crystal structure to semistandard Young tableaux of shape $\lambda$, we exactly recover the crystal $B(\lambda)$ of the irreducible highest weight $U_q(\fsl_n)$-representation of highest weight $\lambda$~\cite{K90,K91}. Furthermore, the crystal operators from Definition~\ref{defn:svt_crystal_ops} also give a crystal structure on words of length $\ell$, which we naturally equate with the tensor product $B(\Lambda_1)^{\otimes \ell}$ (for more details, we refer the reader to~\cite{BS17}).
The Lusztig involution is an involution on highest weight crystals $\ast \colon B(\lambda) \to B(\lambda)$ that sends the highest weight element to the lowest weight element and extended as a crystal isomorphism
\[
e_i(T^*) \mapsto (f_{n+1-i}T)^*,
\qquad
f_i(T^*) \mapsto (e_{n+1-i}T)^*,
\qquad
\wt(T^*) = w_0 \wt(T),
\]
where $w_0$ is the permutation that reverses all entries. We extend this functorially to tensor products by applying it to every factor and then reversing the factors. The Lusztig involution is also given by the Sch\"utzenberger involution (or evacuation) on semistandard tableaux~\cite{Lenart07}.

Recall that for two $U_q(\fsl_n)$-crystals $\mcB, \mcB'$, a \defn{strict crystal morphism} $\psi \colon \mcB \to \mcB'$ is a map $\psi \colon \mcB \sqcup \{0\} \to \mcB \sqcup \{0\}$ such that
\[
\psi(0) = 0,
\qquad
\psi(e_i b) = e_i \psi(b),
\qquad
\psi(f_i b) = f_i \psi(b),
\qquad
\wt\bigl( \psi(b) \bigr) = \wt(b),
\]
where we consider $e_i 0 = 0$ and $f_i 0 = 0$.
We say $\psi$ is an \defn{embedding} (resp.\ \defn{isomorphism}) if $\psi^{-1}(0) = \{0\}$ (resp. $\psi$ is a bijection).
When there exists an isomorphism $\psi \colon \mcB \to \mcB'$, we say $\mcB$ is isomorphic to $\mcB'$ and denote this by $\mcB \iso \mcB'$.

\begin{thm}[{\cite[Thm.~3.9]{MPS18}}]
\label{thm:svt_crystal}
Let $\lambda$ be a partition. Then
\[
\svt^n(\lambda) \iso \bigoplus_{\lambda \subseteq \mu} B(\mu)^{\oplus S_{\lambda}^{\mu}},
\]
where the $S_{\lambda}^{\mu}$ is the ghest weight elements of weight $\mu$ in $\svt^n(\lambda)$.
\end{thm}

For a partition $\lambda$, we will sometimes write it as $\sum_{i=1}^n c_i \Lambda_i$, where $c_i$ denotes the number of columns of height $i$.
This is the usual identification of partitions with the dominant weights associated to $\fsl_n$ using the fundamental weights.

See Figure~\ref{fig:svt_crystal_ex} for an example of the crystal structure on set-valued tableaux.

\begin{figure}
\[
\ytableausetup{boxsize=2.0em}
\begin{tikzpicture}[>=latex,scale=2,every node/.style={scale=0.7},baseline=0]
\node (m) at (0,0) {$\ytableaushort{11,22}$};
\node (f2m) at (0,-1) {$\ytableaushort{11,23}$};
\node (f22m) at (1,-2) {$\ytableaushort{11,33}$};
\node (f12m) at (-1,-2) {$\ytableaushort{12,23}$};
\node (f122m) at (0,-3) {$\ytableaushort{12,33}$};
\node (f1122m) at (0,-4) {$\ytableaushort{22,33}$};
\draw[->,red] (m) -- node[midway,right] {\small $2$} (f2m);
\draw[->,red] (f2m) -- node[midway,above right] {\small $2$} (f22m);
\draw[->,blue] (f2m) -- node[midway, above left] {\small $1$} (f12m);
\draw[->,red] (f12m) -- node[midway,below left] {\small $2$} (f122m);
\draw[->,blue] (f22m) -- node[midway, below right] {\small $1$} (f122m);
\draw[->,blue] (f122m) -- node[midway, right] {\small $1$} (f1122m);
\node (t) at (2,0) {$\ytableaushort{11,2{2,\!3}}$};
\node (f2t) at (2,-1) {$\ytableaushort{11,{2,\!3}3}$};
\node (f12t) at (2,-2) {$\ytableaushort{12,{2,\!3}3}$};
\draw[->,red] (t) -- node[midway,right] {\small $2$} (f2t);
\draw[->,blue] (f2t) -- node[midway, right] {\small $1$} (f12t);
\node (b) at (3,0) {$\ytableaushort{1{1,\!2},23}$};
\node (f2b) at (3,-1) {$\ytableaushort{1{1,\!2},33}$};
\node (f12b) at (3,-2) {$\ytableaushort{{1,\!2}2,33}$};
\draw[->,red] (b) -- node[midway,right] {\small $2$} (f2b);
\draw[->,blue] (f2b) -- node[midway, right] {\small $1$} (f12b);
\node (all) at (4,0) {$\ytableaushort{1{1,\!2},{2,\!3}3}$};
\end{tikzpicture}
\]
\ytableausetup{boxsize=0.8em}
\caption{The $U_q(\fsl_3)$-crystal structure on $\svt^3\left( \ydiagram{2,2} \right)$.}
\label{fig:svt_crystal_ex}
\end{figure}

\subsection{Canonical Grothendieck functions}

From~\cite{Buch02}, we can define a \defn{symmetric Grothendieck function} as
\[
\G_{\lambda}(\xx; \beta) := \sum_{T \in \svt^{\infty}(\lambda)} \beta^{\abs{T}-\abs{\lambda}} \wt(T),
\]
where $\abs{\lambda}$ denotes the size of $\lambda$ (\textit{i.e.}, the number of boxes in $\lambda$).
The value $\abs{T} - \abs{\lambda}$ is the so-called \defn{excess} statistic.
When $\beta = 0$, we recover the \defn{Schur function}:
\[
s_{\lambda}(\xx) = \sum_{T \in \ssyt^{\infty}(\lambda)} \wt(T).
\]
Note that when we restrict to $n$ variables $x_1, x_2, \dotsc, x_n$, we recover the \defn{$\beta$-character} of $\svt^n(\lambda)$ and $\G_{\lambda}$ (and $s_{\lambda}$) is a polynomial (as opposed to a formal power series).
For more on Schur functions, we refer the reader to~\cite[Ch.~7]{ECII}.

The \defn{weak symmetric Grothendieck function} is defined by
\begin{equation}
\label{eq:weak_definition}
\wG_{\lambda}(\xx; \alpha) := \G_{\lambda}\left( \frac{x_1}{1-\alpha x_1}, \frac{x_2}{1-\alpha x_2}, \frac{x_3}{1-\alpha x_3}, \ldots; \alpha \right),
\end{equation}
which recovers the definition given in~\cite[Thm.~6.11]{PP16} when $\alpha = -1$ and $x_i \mapsto -x_i$ (which is for the conjugate shape of the definition in~\cite{LamPyl07}).
Indeed, following~\cite{PP16} we have
\[
\frac{x_i}{1 - \alpha x_i} = x_i + \alpha x_i^2 + \alpha^2 x_i^3 + \cdots = \sum_{k=0}^{\infty} \alpha^k x_i^{k+1},
\]
which is equivalent to allowing multisets to fill the tableaux.
More explicitly, define a \defn{(semistandard) multiset-valued tableau of shape $\lambda$} to be a filling $T$ of the boxes of $\lambda$ by finite nonempty multisets of positive integers such that rows are weakly increasing and columns are strictly increasing in the same sense as for set-valued tableaux.
Let $\mvt^n(\lambda)$ denote the set of all multiset-valued tableaux of shape $\lambda$ and max entry $n$.
Thus, we arrive at the combinatorial definition of~\cite{LamPyl07} for weak symmetric Grothendieck functions:
\[
\wG_{\lambda}(\xx; \alpha) = \sum_{T \in \mvt^{\infty}(\lambda)} \alpha^{\abs{T}-\abs{\lambda}} \wt(T).
\]
A third equivalent way to define a weak symmetric Grothendieck function is by using the involution $\omega$ on symmetric functions given by $\omega s_{\lambda}(\xx) = s_{\lambda'}(\xx)$, where $\lambda'$ is the conjugate partition of $\lambda$.

\begin{prop}[{\cite[Prop.~9.22]{LamPyl07}}]
\label{prop:conjugate_defn}
We have
\[
\wG_{\lambda'}(\xx; \alpha) = \omega \G_{\lambda}(\xx; \alpha).
\]
\end{prop}

Unlike for symmetric Grothendieck functions, we do not obtain a polynomial when we restrict a weak symmetric Grothendieck function to a finite number of variables (\textit{i.e.}, it remains a formal power series).

Symmetric Grothendieck functions and weak symmetric Grothendieck functions have a common generalization given by Yeliussizov~\cite{Yel17}.
A \defn{hook tableau} is a semistandard Young tableau $T$ of the form
\[
\ytableausetup{boxsize=1.8em}
\ytableaushort{h{A_1}{\cdots}{A_k},{L_1},{\raisebox{-2pt}{$\vdots$}},{L_{\ell}}}\,.
\]
We call $h$ the \defn{hook entry} and the entries $A(T) := (A_1, \dotsc, A_k)$ the \defn{arm} and $L(T) := (L_1, \dotsc, L_{\ell})$ the \defn{leg}.
Let $L^+(T) := \{h\} \cup L(T)$ denote the \defn{extended leg}.
A \defn{(semistandard) hook-valued tableau of shape $\lambda$} is a filling $T$ of the boxes of $\lambda$ by hook tableaux such that the rows are weakly increasing and the columns are strictly increasing the in the same sense as for (multi)set-valued tableaux.
Let $\hvt^n(\lambda)$ denote the set of hook-valued tableau of shape $\lambda$ with max entry $n$.
Thus, following~\cite{Yel17}, we define the \defn{canonical Grothendieck polynomial} as
\[
\hG_{\lambda}(\xx; \alpha, \beta) := \sum_{T \in \hvt^{\infty}(\lambda)} \alpha^{\abs{A(T)}} \beta^{\abs{L(T)}} \wt(T).
\]
Note that
\[
\hG_{\lambda}(\xx; \alpha, 0) = \wG_{\lambda}(\xx; \alpha),
\qquad\qquad
\hG_{\lambda}(\xx; 0, \beta) = \G_{\lambda}(\xx; \beta).
\]
Furthermore, Equation~\eqref{eq:weak_definition} follows from~\cite[Prop.~3.4]{Yel17}.

\subsection{Dual canonical Grothendieck functions}

A \defn{reverse plane partition (RPP) of shape $\lambda$} is a filling of $\lambda$ by positive integers such that rows and columns are weakly increasing.
Define the weight of a RRP $P$ to be
\[
\wt(P) := x_1^{m_1} x_2^{m_2} \dotsm x_n^{m_n},
\]
where here $m_i$ is the number of \emph{columns} that contain an $i$ in $P$. As before, denote $\abs{P} := \sum_{i=1}^n m_i$.
Let $\rpp^n(\lambda)$ denote the set of reverse plane partitions with maximum entry $n$.
The \defn{dual symmetric Grothendieck function} $\dG_{\lambda}(\xx; \beta)$ is defined combinatorially by
\[
\dG_{\lambda}(\xx; \beta) = \sum_{P \in \rpp^{\infty}(\lambda)} \beta^{\abs{\lambda} - \abs{P}} \wt(P),
\]
The dual symmetric Grothendieck functions were shown to form a basis Hopf dual to the symmetric Grothendieck functions under the Hall inner product~\cite[Thm.~9.15]{LamPyl07}.
Furthermore, dual symmetric Grothendieck functions are known to be Schur positive~\cite[Thm.~9.8]{LamPyl07}, where $\rpp^n(\lambda)$ was given a $U_q(\fsl_n)$-crystal structure by Galashin~\cite{Galashin17}.

Let $\dwG_{\lambda}(\xx; \alpha)$ denote the \defn{dual weak symmetric Grothendieck function}, which we can define by $\dwG_{\lambda}(\xx; \alpha) = \omega \dG_{\lambda'}(\xx; \alpha)$. The dual weak symmetric Grothendieck functions form the Hopf dual basis of the weak symmetric Grothendieck functions~\cite[Thm.~9.15]{LamPyl07}. This was also given the following combinatorial interpretation~\cite{LamPyl07}. Define a \defn{valued-set tableaux of shape $\lambda$} to be a semistandard Young tableau of shape $\lambda$ such that boxes within a particular row are divided into \defn{groups}. Note that our description is conjugate to that from~\cite{LamPyl07}.
Define the weight of a valued-set tableau $V$ to be
\[
\wt(V) := x_1^{m_1} x_2^{m_2} \dotsm x_n^{m_n},
\]
where here $m_i$ is the number of \emph{groups} that contain an $i$ in $V$. As before, denote $\abs{P} := \sum_{i=1}^n m_i$.
Thus, the dual weak symmetric Grothendieck function can be given by
\[
\dwG_{\lambda}(\xx; \alpha) = \sum_{V \in \vst^{\infty}(\lambda)} \alpha^{\abs{\lambda} - \abs{V}} \wt(V),
\]
where $\vst^{\infty}(\lambda)$ is the set of all valued-set tableaux of shape $\lambda$ with max entry $n$.

We also require some additional definitions on valued-set tableaux in the sequel.
We call the leftmost (resp.\ rightmost) entry in a group the \defn{buoy} (resp.\ \defn{anchor}).
Thus, $m_i$ in the weight is also equal to the number of buoys $i$ in a valued-set tableau $V$ and $\abs{V}$ is the number of buoys (equivalently, anchors or groups).
We will also consider groups constructed by adding a vertical \defn{divider} between certain pairs of entries $i$ in the same row.

\subsection{Stembridge axioms for crystals}

We recall the Stembridge axioms~\cite{Stembridge03}, a set of local criteria used to determine if a $U_q(\fsl_n)$-crystal is isomorphic to a direct sum of highest weight crystals, but instead given the crystal and the crystal operators rather than the crystal graph. Recall that the Cartan matrix for $\fsl_n$ is given by
\[
(A_{ij})_{i,j=1}^{n-1} = \begin{bmatrix}
2 & -1 \\
-1 & 2 & -1 \\
& \ddots & \ddots & \ddots \\
& & -1 & 2 & -1 \\
& & & -1 & 2
\end{bmatrix}.
\]

Let $B$ be a set with crystal operators $e_i, f_i \colon B \to B \sqcup \{0\}$ such that $e_i$ and $f_i$ are computed using an $i$-signature; that is to say, $e_i$ (resp.~$f_i$) changes the rightmost $\bminus$ (resp.~leftmost $\bplus$) into a $\bplus$ (resp.~$\bminus$). Furthermore, every contribution of a $\bplus$ (resp.~$\bminus$) corresponds to multiplying $x_i$ (resp.~$x_{i+1}$) to the weight.
Define an \defn{$i$-string} as a subset of $B$ closed under $e_i$ and $f_i$.
We require the following statistics
\begin{align*}
\delta_i(b) & := -\max \{ k \in \ZZ_{\geq 0} \mid e_i^k b \neq 0 \},
&
\varphi_i(b) & := \max \{ k \in \ZZ_{\geq 0} \mid f_i^k b \neq 0 \},
\\
\Delta_i \delta_j(b) & := \delta_j(e_i b) - \delta_j(b),
&
\Delta_i \varphi_j(b) & := \varphi_j(e_i b) - \varphi_j(b),
\\
\nabla_i \delta_j(b) & := \delta_j(b) - \delta_j(f_i b),
&
\nabla_i \varphi_j(b) & := \varphi_j(b) - \varphi_j(f_i b).
\end{align*}

The following are the \defn{Stembridge axioms}:
\begin{itemize}
\item[(P1)] All $i$-strings have no cycles (\textit{i.e.}, there does not exist a $b$ such that $f_i^k b = b$ for some $k > 0$) and finite length.

\item[(P2)] For any $b, b' \in B$, we have $e_i b = b'$ if and only if $b = f_i b'$.

\item[(P3)] $\Delta_i \delta_j(b) + \Delta_i \varphi_j(b) = A_{ij}$.

\item[(P4)] $\Delta_i \delta_j(b) \leq 0$ and $\Delta_i \varphi_j(b) \leq 0$.

\item[(P5)] If $e_i b, e_j b \neq 0$, then $\Delta_i \delta_j(b) = 0$ implies $y := e_i e_j b = e_j e_i b$ and $\nabla_j \varphi_i(y) = 0$.

\item[(P6)] If $e_i b, e_j b \neq 0$, then $\Delta_i \delta_j(b) = \Delta_j \delta_i(b) = -1$ implies $y := e_i e_j^2 e_i b = e_j e_i^2 e_j b$ and $\nabla_i \varphi_j(y) = \nabla_j \varphi_i(y) = -1$.

\item[(P5')] If $f_i b, f_j b \neq 0$, then $\nabla_i \phi_j(b) = 0$ implies $y := f_i f_j b = f_j f_i b$ and $\Delta_j \delta_i(y) = 0$.

\item[(P6')] If $f_i b, f_j b \neq 0$, then $\nabla_i \phi_j(b) = \nabla_j \phi_i(b) = -1$ implies $y := f_i f_j^2 f_i b = f_j f_i^2 f_j b$ and $\Delta_i \delta_j(y) = \Delta_j \delta_i(y) = -1$.
\end{itemize}

\begin{thm}[{\cite[Thm.~3.3]{Stembridge03}}]
\label{thm:stembridge}
Let $B$ be a crystal that satisfies the Stembridge axioms such that every connected component $C^{(i)}$ contains a highest weight element $u^{(i)}$ of weight $\lambda^{(i)}$, then
\[
B \iso \bigoplus_{i=1}^N B(\lambda^{(i)}).
\]
\end{thm}

We note that~(P3) and~(P4) are equivalent to one of three possibilities:
\[
\bigl( A_{ij}, \Delta_i \delta_j(b), \Delta_i \varphi_j(b) \bigr) = (0,0,0), (-1, -1, 0), (-1, 0, -1).
\]
From the weight condition of the $i$-signature, we have $\phi_i(b) + \delta_i(b) = m_i - m_{i+1}$, where $\wt(b) = x_1^{m_1} \dotsm x_n^{m_n}$. So along with~(P2), $B$ with the crystal operators satisfy the abstract $U_q(\fsl_n)$-crystal axioms~\cite{K93}.

\section{Multiset-valued tableaux}
\label{sec:crystal_MVT}

In this section, we prove our first main result: there exists a $U_q(\fsl_n)$-crystal structure on $\mvt^n(\lambda)$ such that it is isomorphic to a direct sum of irreducible highest weight crystals $B(\mu)$.
After that, we discuss the relation with the usual crystal structure on semistandard Young tableaux and some consequences for stable dual Grothendieck polynomials.

\subsection{Crystal structure}

In order to define the crystal structure, we start by defining the crystal operators $e_i, f_i \colon \mvt^n(\lambda) \to \mvt^n(\lambda) \sqcup \{0\}$.
We do so by following the signature rule as for set-valued tableaux in Definition~\ref{defn:svt_crystal_ops}, for which we need to define an appropriate reading word.
We write multisets as words for compactness.

\begin{dfn}[Reading word]
Let $T \in \mvt^n(\lambda)$.
Let $C$ be a column of $T$. Define the \defn{column reading word} $\rd(C)$ by first reading the smallest entry of each box from bottom-to-top in $C$, then reading the remaining entries from smallest to largest in each box from top-to-bottom in $C$.
Define the \defn{reading word} $\rd(T) = \rd(C_1) \rd(C_2) \dotsm \rd(C_k)$, where $C_1, C_2, \dotsc, C_k$ are the columns of $T$ from left-to-right.
\end{dfn}

\begin{ex}
For the multiset-valued (column) tableau
\[
\ytableausetup{boxsize=2.2em}
C = \; \ytableaushort{{\ml{1}13},{\ml{4}445},{\ml{6}},{\ml{7}899}}\,,
\]
we have
\[
\rd(C) = \mathbf{\color{darkred}7641} 13 445 899.
\]
\end{ex}

\begin{dfn}[Crystal operators]
\label{defn:mvt_crystal_ops}
Fix $T \in \mvt^n(\lambda)$ and $i \in I$.
Write $\bplus$ for each $i$ in $\rd(T)$ and $\bminus$ for each $i+1$ in $\rd(T)$ (ignore all other letters). 
Next, cancel signs in ordered pairs $\bminus \bplus$ until obtaining a sequence of the form $\bplus \cdots \bplus \bminus \cdots \bminus$ called the \defn{$i$-signature}.
\begin{description}
\item[\defn{$e_i T$}] If there is no $\bminus$ in the resulting sequence, then $e_i T = 0$. Otherwise let $\bbb$ correspond to the box of the leftmost uncanceled $\bminus$. Then $e_i T$ is given by one of the following:
\begin{itemize}
\item if there exists a box $\bbb^{\uparrow}$ immediately above $\bbb$ that contains an $i$, then remove $i + 1$ from $\bbb$ and add $i$ to $\bbb^{\uparrow}$;
\item otherwise replace the $i+1$ in $\bbb$ with an $i$.
\end{itemize}

\item[\defn{$f_i T$}] If there is no $\bplus$ in the resulting sequence, then $f_i T = 0$. Otherwise let $\bbb$ correspond to the box of the rightmost uncanceled $\bplus$. Then $f_i T$ is given by one of the following:
\begin{itemize}
\item if there exists a box $\bbb^{\downarrow}$ immediately below $\bbb$ that contains an $i+1$, then remove the $i$ from $\bbb$ and add an $i+1$ to $\bbb^{\downarrow}$;
\item otherwise replace the $i$ in $\bbb$ with an $i+1$.
\end{itemize}
\end{description}
\end{dfn}

\begin{ex}
The connected components in $\mvt^3(\Lambda_2)$ with the crystal operators from Definition~\ref{defn:mvt_crystal_ops} that correspond to $\alpha^0$, $\alpha^1$, and $\alpha^2$ are
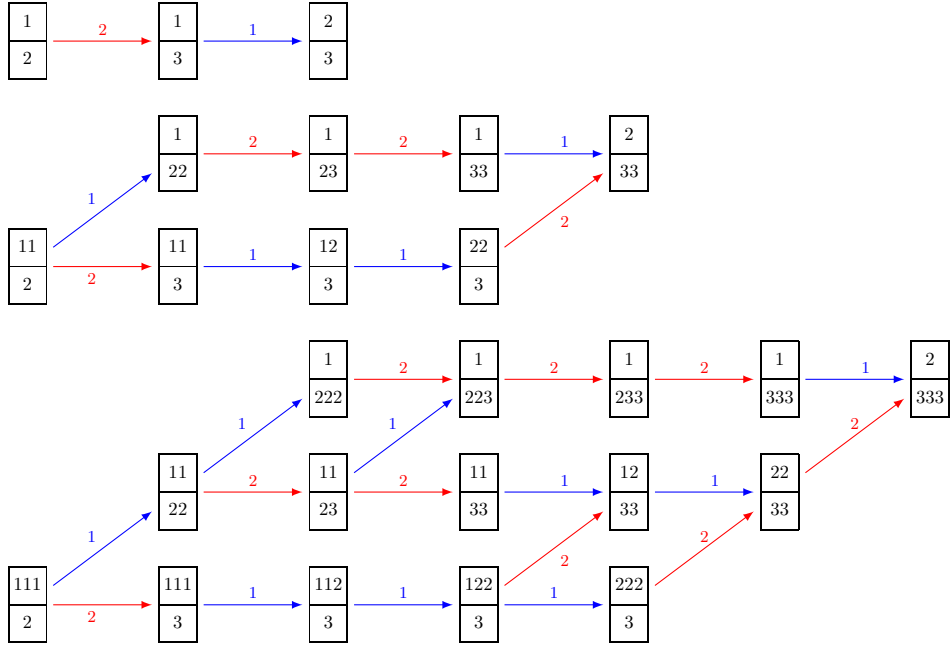
\[
\ytableausetup{boxsize=2.0em}
\begin{tikzpicture}[>=latex,scale=2,every node/.style={scale=0.7}]
\node (c0) at (0,0) {$\ytableaushort{1,2}$};
\node (c1) at (1,0) {$\ytableaushort{1,3}$};
\node (c2) at (2,0) {$\ytableaushort{2,3}$};
\draw[->,red] (c0) -- node[midway,above] {\small 2} (c1);
\draw[->,blue] (c1) -- node[midway,above] {\small 1} (c2);
\begin{scope}[yshift=-1.5cm]
\node (s0) at (0,0) {$\ytableaushort{{11},2}$};
\node (s1) at (1,0.75) {$\ytableaushort{1,{22}}$};
\node (s2) at (1,0) {$\ytableaushort{{11},3}$};
\node (s12) at (2,0) {$\ytableaushort{{12},3}$};
\node (s21) at (2,0.75) {$\ytableaushort{1,{23}}$};
\node (s112) at (3,0) {$\ytableaushort{{22},3}$};
\node (s221) at (3,0.75) {$\ytableaushort{1,{33}}$};
\node (sb) at (4,0.75) {$\ytableaushort{2,{33}}$};
\draw[->,red] (s0) -- node[midway,below left] {\small 2} (s2);
\draw[->,blue] (s0) -- node[midway,above left] {\small 1} (s1);
\draw[->,red] (s1) -- node[midway,above] {\small 2} (s21);
\draw[->,blue] (s2) -- node[midway,above] {\small 1} (s12);
\draw[->,red] (s21) -- node[midway,above] {\small 2} (s221);
\draw[->,blue] (s12) -- node[midway,above] {\small 1} (s112);
\draw[->,red] (s112) -- node[midway,below right] {\small 2} (sb);
\draw[->,blue] (s221) -- node[midway,above right] {\small 1} (sb);
\end{scope}
\begin{scope}[yshift=-3.75cm]
\node (s0) at (0,0) {$\ytableaushort{{111},2}$};
\node (s1) at (1,0.75) {$\ytableaushort{{11},{22}}$};
\node (s2) at (1,0) {$\ytableaushort{{111},3}$};
\node (s12) at (2,0) {$\ytableaushort{{112},3}$};
\node (s21) at (2,0.75) {$\ytableaushort{{11},{23}}$};
\node (s112) at (3,0) {$\ytableaushort{{122},3}$};
\node (s221) at (3,0.75) {$\ytableaushort{{11},{33}}$};
\node (s1221) at (4,0.75) {$\ytableaushort{{12},{33}}$};
\node (s11) at (2,1.5) {$\ytableaushort{1,{222}}$};
\node (s211) at (3,1.5) {$\ytableaushort{1,{223}}$};
\node (s2211) at (4,1.5) {$\ytableaushort{1,{233}}$};
\node (s22211) at (5,1.5) {$\ytableaushort{1,{333}}$};
\node (s1112) at (4,0) {$\ytableaushort{{222},3}$};
\node (s21112) at (5,0.75) {$\ytableaushort{{22},{33}}$};
\node (sb) at (6,1.5) {$\ytableaushort{2,{333}}$};
\draw[->,red] (s0) -- node[midway,below left] {\small 2} (s2);
\draw[->,blue] (s0) -- node[midway,above left] {\small 1} (s1);
\draw[->,red] (s1) -- node[midway,above] {\small 2} (s21);
\draw[->,blue] (s2) -- node[midway,above] {\small 1} (s12);
\draw[->,red] (s21) -- node[midway,above] {\small 2} (s221);
\draw[->,blue] (s12) -- node[midway,above] {\small 1} (s112);
\draw[->,blue] (s112) -- node[midway,above] {\small 1} (s1112);
\draw[->,red] (s112) -- node[midway,below right] {\small 2} (s1221);
\draw[->,blue] (s221) -- node[midway,above right] {\small 1} (s1221);
\draw[->,blue] (s1) -- node[midway,above left] {\small 1} (s11);
\draw[->,red] (s11) -- node[midway,above] {\small 2} (s211);
\draw[->,red] (s211) -- node[midway,above] {\small 2} (s2211);
\draw[->,red] (s2211) -- node[midway,above] {\small 2} (s22211);
\draw[->,blue] (s22211) -- node[midway,above right] {\small 1} (sb);
\draw[->,red] (s1112) -- node[midway,above] {\small 2} (s21112);
\draw[->,blue] (s1221) -- node[midway,above right] {\small 1} (s21112);
\draw[->,blue] (s21) -- node[midway,above left] {\small 1} (s211);
\draw[->,red] (s21112) -- node[midway,above] {\small 2} (sb);
\end{scope}
\end{tikzpicture}
\]
\end{ex}

First we show that the crystal operators are well-defined and satisfy the requisite properties.

\begin{lemma}
\label{lemma:fixed_reading}
Let $T \in \mvt^n(\lambda)$ and suppose $f_i T \neq 0$, then the $i$ changed to $i+1$ in $f_i T$ does not change its position in the reading word.
That is to say, we have $\rd(f_i T) = f_i \rd(T)$.
\end{lemma}

\begin{proof}
Let $T' = f_i T$, and suppose the changed $i$ is in box $\bbb$.
Clearly if $i$ becomes an $i+1$ in the same box, then it does not change its position.
Now consider when $i$ moves from $\bbb$ and becomes an $i+1$ in the box $\bbb^{\downarrow}$ immediately below $\bbb$. In this case, the $i$ changed is the rightmost $i$ in the $i$-signature of $T$, and hence the rightmost entry of $\bbb$ in the reading word (we also have $\max \bbb = i$ since $\min \bbb^{\downarrow} = i+1$). Hence, the changed $i+1$ in $T'$ must be read immediately after all entries of $\bbb^{\downarrow}$ in $T'$ have been read (consider this added $i+1$ to be the second such letter), which means it remains at the same position.
\end{proof}

\begin{remark}
We note that our reading word is the column version of the reading word from~\cite[Def.~2.5]{BM12}. Furthermore, when we consider the reading word from~\cite[Def.~2.5]{BM12} applied to $\svt^n(\lambda)$, but otherwise keep the same crystal operators, then the analog of Lemma~\ref{lemma:fixed_reading} holds in that setting.
In addition, our reading word and crystal structure for a single column is similar to the one for the minimaj crystal from~\cite{BCHOPSY18}.
\end{remark}

\begin{lemma}
\label{lemma:mvt_ops_well_defined}
Let $T, T' \in \mvt^n(\lambda)$. Then
\[
e_i T \in \mvt^n(\lambda) \sqcup \{0\},
\qquad
f_i T' \in \mvt^n(\lambda) \sqcup \{0\},
\qquad
e_i T = T' \Longleftrightarrow T = f_i T'.
\]
\end{lemma}

\begin{proof}
We first show $f_i T' \in \mvt^n(\lambda) \sqcup \{0\}$, and if $f_i T' = 0$, then the claim is trivially true.
Suppose $T = f_i T'$ is formed by changing an $i$ to an $i+1$ within the same box $\bbb$. To show $T \in \mvt^n(\lambda)$, we first note that we could not have an $i+1$ in the box $\bbb^{\downarrow}$ immediately below $\bbb$. So it remains to show that there does not exist an $i$ in the box $\bbb^{\rightarrow}$ immediately to the right of $\bbb$. By semistandardness of $T'$, for the box $\bbb^{\searrow}$ immediately to the right of $\bbb^{\downarrow}$, we must have $i + 1 \notin \bbb^{\searrow}$. Hence, the $i \in \bbb^{\rightarrow}$ must correspond to an unpaired $\bplus$ to the right of the $\bplus$ from $\bbb$, which is a contradiction. Thus, we have $T \in \mvt^n(\lambda)$ in this case.

Now instead suppose $T = f_i T'$ is formed by moving an $i$ from a box $\bbb$ to an $i+1$ in $\bbb^{\downarrow}$, the box immediately below $\bbb$. Since $i+1 \in \bbb^{\downarrow}$ already in $T'$, we have $T \in \mvt^n(\lambda)$.

Next we show the claim $e_i T = T'$ if and only if $T = f_i T'$.
From Lemma~\ref{lemma:fixed_reading} (and the analogous statement for $e_i$), the $i$-signatures of $T'$ and $T$ must differ by $\bplus \leftrightarrow \bminus$ corresponding to the $i \leftrightarrow i+1$ with no other cancellations. Hence $e_i$ changes $i+1 \mapsto i$ in the same box if and only if $f_i$ changes $i \mapsto i+1$ in the same box. Therefore, we have $e_i T = T'$ if and only if $T = f_i T'$.

Finally, the claim $e_i T \in \mvt^n(\lambda)$ follows from the other two statements (this claim could also be proven directly similar to $f_i T' \in \mvt^n(\lambda)$).
\end{proof}

Next we show that the coefficient of $\alpha^a$ for a weak symmetric Grothendieck of a single column is isomorphic to a crystal of semistandard tableaux for hook shape.
Note that $\abs{T} - k = \abs{T} - \abs{\lambda}$.

\begin{prop}
\label{prop:column_isomorphism}
Recall that $1^k = \Lambda_k$ is a single column of height $k$.
Let
\[
\mvt^n_a(\Lambda_k) := \{ T \in \mvt^n(\Lambda_k) \mid \abs{T} - k = a \}.
\]
Then
\[
\mvt^n_a(\Lambda_k) \iso B(\Lambda_k + a\Lambda_1).
\]
\end{prop}

\begin{proof}
Let $\mu = \Lambda_k + a \Lambda_1$.
We prove the claim by constructing an explicit crystal isomorphism $\psi \colon \mvt^n_a(\Lambda_k) \to B(\mu)$.
We define $\psi$ by
\[
\ytableausetup{boxsize=2.3em}
\def\arraystretch{1.9}
\begin{array}{|@{\;}c@{\;}|}
\hline
m_1 \leq o_1 \leq \cdots \leq o_{a_1}
\\\hline
m_2 \leq o_{a_1+1} \leq \cdots \leq o_{a_2}
\\\hline
\vdots
\\\hline
m_k \leq o_{a_{k-1}+1} \leq \cdots \leq o_{a_k}
\\\hline
\end{array}
\; \mapsto \;
  \ytableaushort{{m_1}{o_1}{\cdots}{o_{a_1}}{o_{a_1\!+1}}{\cdots}{o_{a_2}}{\cdots}{o_{a_k}},{m_2},{\raisebox{-2pt}{$\vdots$}},{m_k}}\,,
\]
We note that $m_i \leq o_{a_i} < m_{i+1} \leq o_{a_i+1}$ since the column is strictly increasing. Hence, we have $\psi(T) \in \ssyt^n(\mu)$ for all $T \in \mvt^n(\Lambda_k)$.
Note also that the reading words of $T$ and $\psi(T)$ are equal and so $\psi$ is a crystal isomorphism by Lemma~\ref{lemma:fixed_reading}.
\end{proof}

\begin{thm}
\label{thm:mvt_crystal}
Let $\lambda$ be a partition.
For any $T \in \mvt^n(\lambda)$ such that $T$ is a highest weight element, the closure of $T$ under the crystal operators is isomorphic to $B(\mu)$, where $\mu = \wt(T)$.
Moreover, we have
\[
\mvt^n(\lambda) \iso \bigoplus_{\mu \supseteq \lambda} B(\mu)^{\oplus M_{\lambda}^{\mu}},
\]
where $M_{\lambda}^{\mu}$ is the number of highest weight elements of weight $\mu$ in $\mvt^n(\lambda)$, and
\[
\wG_{\lambda}(\xx; \alpha) = \sum_{\mu \supseteq \lambda} \alpha^{\abs{\mu}-\abs{\lambda}} M_{\lambda}^{\mu} s_{\mu}(\xx).
\]
\end{thm}

\begin{proof}
Note that we can consider any multiset-valued tableau as a tensor product of single column multiset-valued tableaux by the definition of the crystal operators and the reading word.
In particular, this gives a strict crystal embedding, which implies that the image is a union of connected components.
Hence the first claim follows from Proposition~\ref{prop:column_isomorphism}, Lemma~\ref{lemma:mvt_ops_well_defined}, and that the tensor product of highest weight crystals is a direct sum of highest weight crystals.
The other two claims follow immediately from the first.
\end{proof}

We remark that our proof technique is similar to that used in~\cite{MPS18} in that we show the isomorphism for the fundamental building blocks, here these are single columns (Proposition~\ref{prop:column_isomorphism}), and using general properties of the tensor product rule (in~\cite{MPS18}, the building blocks were single rows).
We also note that Lemma~\ref{lemma:fixed_reading} and Lemma~\ref{lemma:mvt_ops_well_defined} also immediately yields Theorem~\ref{thm:mvt_crystal} as every $B(\mu)$ is a subcrystal of $B(\Lambda_1)^{\otimes \abs{\mu}}$, where the strict embedding is given by the reading word.

\begin{prop}
\label{prop:hw_mst}
Suppose $T \in \mvt^n(\lambda)$ is a highest weight element.  Then the $i$-th row of $T$ contains only instances of the letter $i$.
\end{prop}

\begin{proof}
It is sufficient to show this for the rightmost column $C$ as the claim follows for $T$ by semistandardness (in particular, that the rows must be weakly increasing). Suppose the claim is false: there exists a highest weight element $T \in \mvt^n(\lambda)$ such that there exists an $m$ in a box $\bbb$ in $C$ at row $r > m$. Let $m$ be minimal such value. By the signature rule for $e_{m-1}$, any $\bplus$ from an $m-1$ that would cancel the $\bminus$ from such an $m$ must occur later in the reading word. If this $m$ is not the smallest entry in $\bbb$ or the multiplicity of $m$ in $\bbb$ is at least $2$, then there is no $m-1$ after it in the reading word by the column strictness of $C$ and that $C$ is the rightmost column in $T$. Hence, we have $e_{m-1} T \neq 0$, which is a contradiction. Otherwise, in order to be highest weight, there was an $m-1$ in row $r-1$, we would have $m - 1 > r - 1$. However, this contradicts the minimality of $m$, and so the claim follows.
\end{proof}

\begin{cor}
\label{cor:same_length_mst}
For any $\mu \supseteq \lambda$ such that $B(\mu) \subseteq \mvt^n(\lambda)$, we have $\ell(\mu) = \ell(\lambda)$.
\end{cor}

\begin{proof}
This follows immediately from Proposition~\ref{prop:hw_mst} as no $k > i$ can appear in any highest weight element of $\mvt^n(\lambda)$.
\end{proof}

\begin{cor}
\label{cor:hw_mst_skew_ssyt}
The number of highest weight multiset-valued tableaux of shape $\lambda$ and weight $\mu$ is the number of semistandard Young tableaux of shape $\mu / \lambda$ such that all of the entries of the $i$-th row are in the (closed) interval $[\lambda_1+1-\lambda_i,\lambda_1]$.
\end{cor}

\begin{proof}
Let $\mu = (\mu_1, \mu_2, \dotsc, \mu_{\ell})$. By Proposition~\ref{prop:hw_mst}, we have $\ell = r$ as we can construct any highest weight element $T \in \mvt^n(\lambda)$ of weight $\mu$ by choosing the number of additional entries $x_{ij} \geq 0$ of $i$ in each box at position $(i,j) \in \lambda$ such that $\sum_{1 \leq j \leq k} x_{ij} = \mu_i - \lambda_i$ for all $1 \leq i \leq r$. However, not all of these elements will be highest weight by the signature rule; in particular, we require for each $(i,j) \in \lambda$ that
\begin{equation}
\label{eq:bracketing_condition}
\lambda_{i+1} + \sum_{j \leq k \leq \lambda_{i+1}} x_{i+1,k} \leq \lambda_i + \sum_{j+1 \leq k \leq \lambda_i} x_{ik}.
\end{equation}
Note that this is equivalent to choosing a semistandard Young tableau of shape $\mu / \lambda$ with the entries of row $i$ being in the interval $[\lambda_1 + 1 - \lambda_i, \lambda_1]$ by considering the $k$-th extra $j$ read from right to left in column $\lambda_1 - i$ to be an $i$ in row $j$ of the semistandard Young tableau.
Indeed, we cannot have an entry in the interval $[1, \lambda_1 - \lambda_i]$ and the bracketing condition~\eqref{eq:bracketing_condition} is equivalent to the column strictness of the Young tableau and the skew shape.
\end{proof}

\begin{ex}
Consider $\lambda = 32 = 2\Lambda_2 + \Lambda_1$ and $\mu = 44 = 4 \Lambda_2$. Then the following are the highest weight MVT and their semistandard tableaux under Corollary~\ref{cor:hw_mst_skew_ssyt}:
\[
\ytableausetup{boxsize=1.8em}
\begin{array}{c@{\hspace{30pt}}c}
\ytableaushort{11{11},2{222}} \rightarrow \ytableaushort{\none1,22}
&
\ytableaushort{11{11},{22}{22}} \rightarrow \ytableaushort{\none1,23}
\\[1cm]
\ytableaushort{11{11},{222}2} \rightarrow \ytableaushort{\none1,33}
&
\ytableaushort{1{11}1,{222}2} \rightarrow \ytableaushort{\none2,33}
\end{array}
\]
\end{ex}

We note that if $\lambda = k^r = k \Lambda_r$ is an $r \times k$ rectangle, the Corollary~\ref{cor:hw_mst_skew_ssyt} implies that the multiplicity is $\abs{\ssyt^k(\mu/\lambda)}$ since $[\lambda_1 + 1 - \lambda_i, \lambda_1] = [k + 1 - k, k] = [1, k]$ is a constant. Furthermore, in this case we can consider $\mu/\lambda$ as a straight shape given by $(\mu_1 - k, \mu_2 - k, \dotsc, \mu_r - k)$, where $\mu = (\mu_1, \mu_2, \dotsc, \mu_r)$, and so we can count them by the hook content formula (see, \textit{e.g.},~\cite[Thm.~7.21.2]{ECII}).

\begin{ex}
\label{ex:rectangle_decomp}
Let $\lambda = 33 = 3\Lambda_2$ and $\mu = 54 = 4\Lambda_2 + \Lambda_1$. The highest weight multiset valued tableaux of shape $\lambda$ and weight $\mu$ are in bijection with the set of SSYT of shape $21$ in the letters $\{1,2,3\}$:
\begin{align*}
\ytableausetup{boxsize=1.8em}
     \ytableaushort{{1}{1}{111},{2}{22}{2}} & \rightarrow 
  \ytableaushort{11,2} &
     \ytableaushort{{1}{1}{111},{22}{2}{2}} & \rightarrow 
  \ytableaushort{11,3}
\\[5pt]
     \ytableaushort{{1}{11}{11},{2}{22}{2}} & \rightarrow 
  \ytableaushort{12,2} &
     \ytableaushort{{1}{11}{11},{22}{2}{2}} & \rightarrow 
  \ytableaushort{12,3}
\allowdisplaybreaks\\[5pt]
     \ytableaushort{{11}{1}{11},{2}{22}{2}} & \rightarrow 
  \ytableaushort{13,2} &
     \ytableaushort{{11}{1}{11},{22}{2}{2}} & \rightarrow 
  \ytableaushort{13,3}
\\[5pt]
     \ytableaushort{{1}{1}{111},{2}{22}{2}} & \rightarrow 
  \ytableaushort{22,3} &
     \ytableaushort{{11}{11}{1},{22}{2}{2}} & \rightarrow 
  \ytableaushort{23,3}
\end{align*}
\end{ex}

\subsection{Uncrowding the crystal structure}
\label{sec:uncrowding}

An \defn{increasing tableau} is a semistandard Young tableau that is also strictly increasing across its rows.
Let $\mcF^c_{\mu/\lambda}$ denote the set of increasing tableaux of shape $\mu/\lambda$ where the $i$-th column is \defn{strictly flagged} by $i$, that is to say the maximum entry in the $i$-th column is strictly less than $i$.
We call a tableau in $\mcF^c_{\mu/\lambda}$ a \defn{column flagged tableau}.
Let $T \xleftarrow{\RSK} T'$ denote the Robinson--Schensted--Knuth (RSK) insertion (see, \textit{e.g.},~\cite{ECII} for more details on RSK) of the reading word of $T'$ into $T$.

Next, we construct an explicit crystal isomorphism
\[
\Upsilon \colon \mvt^n(\lambda) \to \bigsqcup_{\mu \supseteq \lambda} B(\mu) \times \mcF^c_{\mu/\lambda},
\]
where the crystal structure on the codomain is given by $f_i(b \times F) = (f_i b) \times F$ for all $b \times F \in B(\mu) \times \mcF^c_{\mu/\lambda}$ for any fixed $\mu$.
We call the map $\Upsilon$ \defn{uncrowding} as it is given similar to the uncrowding map for set-valued tableaux (see~\cite[Sec.~6]{Buch02},~\cite[Sec.~5]{BM12}, and~\cite[Thm.~3.12]{MPS18}; see also~\cite{RTY18}), but working column-by-column and measuring the growth of the diagram along columns.
More specifically, for any $T \in \mvt^n(\lambda)$ we define $\Upsilon(T)$ recursively starting with $b_{\lambda_1+1} \times F_{\lambda_1+1} = \emptyset \times \emptyset$. Suppose we are at step $i$ with the current state being $b_i \times F_i$, and let $C_j$ denote the $j$-th column of $T$. Construct
\[
b_{i-1} := \rd(C_{i-1}) \xleftarrow{\RSK} \rd(C_i) \xleftarrow{\RSK} \cdots \xleftarrow{\RSK} \rd(C_{\lambda_1}).\footnote{This is equivalent to RSK inserting the image of (the reading word of) the image under the isomorphism $\psi$ from Proposition~\ref{prop:column_isomorphism} of the rightmost $i-1$ columns of $T$.}
\]
Construct $F_{i-1}$ by starting first with $F_i$ of shape $\mu_i$ but shifting the necessary elements to the right one step, so partially filling in the shape $\mu_{i-1} / \lambda_{\geq i-1}$, where
\begin{equation}
\label{eq:rightmost_column_shape}
\lambda_{\geq i-1} = \bigl( \max(\lambda_1-i+2,0), \dotsc, \max(\lambda_{\ell}-i+2,0) \bigr)
\end{equation}
is the shape of the rightmost $i-1$ columns of $\lambda$ and $\mu_{i-1}$ is the shape of $b_{i-1}$.
Then add entries in the unfilled boxes in column $j$ with entry $j-1$ until $F_{i-1}$ has been filled in. Thus, we constructed the $(i-1)$-th step $b_{i-1} \times F_{i-1}$. Repeating this for every column, the final result is $\Upsilon(T) = b_1 \times F_1$.

\begin{ex}
Applying uncrowding to
\[
\ytableausetup{boxsize=1.8em}
T = \ytableaushort{{11}11,2{22},{33}}\,,
\]
\ytableausetup{boxsize=1.5em}
we first start with $b_4 \times F_4 = \emptyset \times \emptyset$. We then RSK insert the reading word $1$ of the rightmost column and obtain $b_3 \times F_3 = \ytableaushort{1} \times \ytableaushort{{\ml{1}}}$. Next, we consider the insertion tableau under RSK of $\ml{212}1$ and obtain
\[
b_2 \times F_2 = \ytableaushort{11,22} \quad \times \quad \ytableaushort{{\cdot}{\cdot},{\cdot}{\ml{1}}}\,.
\]
Finally, to obtain $\Upsilon(T) = b_1 \times F_1$, we perform RSK on $\ml{32113}2121$ to obtain
\[
b_1 \times F_1 = \ytableaushort{1111,222,33} \quad \times \quad \ytableaushort{{\cdot}{\cdot}{\cdot}{\ml{3}},{\cdot}{\cdot}{1},{\cdot}{\ml{1}}}\,,
\]
where the $1$ in the third column comes from $F_2$ shifted to the right one step.
\end{ex}

\begin{ex}
\label{ex:uncrowding_big}
Applying uncrowding to
\[
\ytableausetup{boxsize=1.8em}
T = \ytableaushort{{112}{22}{256},{33}{444}{7},{568},{9}}\,,
\]
where $\rd(T) = 953112368 \; 42244 \; 7256$,
we obtain
\begin{align*}
\ytableausetup{boxsize=1.5em}
b_4 \times F_4 & = \emptyset & \times \quad & \emptyset,  
\\ b_3 \times F_3 & = \ytableaushort{256,7} & \times \quad & \ytableaushort{{\cdot}{\ml{1}}{\ml{2}},{\cdot}}\,,  
\allowdisplaybreaks \\ b_2 \times F_2 & = \ytableaushort{222456,447} & \times \quad & \ytableaushort{{\cdot}{\cdot}12{\ml{4}}{\ml{5}},{\cdot}{\cdot}{\ml{2}}}\,,  
\allowdisplaybreaks \\ b_1 \times F_1 & = \ytableaushort{112222456,33447,568,9} & \times \quad & \ytableaushort{{\cdot}{\cdot}{\cdot}1245{\ml{7}}{\ml{8}},{\cdot}{\cdot}{\cdot}2{\ml{4}},{\cdot}{\ml{1}}{\ml{2}},{\cdot}}\,,
\end{align*}
resulting in $\Upsilon(T) = b_1 \times F_1$.
\end{ex}

\begin{thm}
\label{thm:flagged_decomposition}
We have
\[
\mvt^n(\lambda) \iso \bigoplus_{\mu \supseteq \lambda} B(\mu)^{\oplus \abs{\mcF^c_{\mu/\lambda}}},
\]
where the isomorphism is given by the uncrowding map $\Upsilon$, and so $M_{\lambda}^{\mu} =  \abs{\mcF^c_{\mu/\lambda}}$. Moreover, we have
\[
\wG_{\lambda}(\xx; \alpha) = \sum_{\mu \supseteq \lambda} \alpha^{\abs{\mu}-\abs{\lambda}} \abs{\mcF^c_{\mu/\lambda}} s_{\mu}(\xx).
\]
\end{thm}

\begin{proof}
It is sufficient to show that the map $\Upsilon$ is a bijection as RSK insertion is a crystal isomorphism (see, \textit{e.g.},~\cite{BS17,LLT02}). By the construction, the result of $\Upsilon$ satisfies the flagging and row strictness conditions. By the standard properties of RSK and a straightforward induction, the result of $\Upsilon$ also satisfies the column strictness condition. So the map $\Upsilon$ is well-defined. We can construct the inverse $\Upsilon^{-1}$ by conjugating by the Lusztig involution and recursively by applying reverse RSK inserting the boxes in the column flagged tableau that are maximal in their column and the extra column added. Indeed, if we are at $b_i \times F_i$, we first apply the Lusztig involution to $b_i$ to get $b_i^*$ and perform reverse RSK insertion on the boxes of $b_i^*$ on the outer corners from bottom to top, and then on the outer box in the row corresponding to a box $\bbb$ in column $j$ of $F$ such that $j-1$ is the entry in $\bbb$, doing this from left-to-right until no longer possible. That is to say we treat all of these extra entries as being the same value in the recording tableaux. This determines the $i$-th column after applying the Lusztig involution again to the resulting word. Since $\RSK(w^*)^* = \RSK(w)$, the result is $b_{i+1} \times F_{i+1}$. Thus this is the inverse procedure of $\Upsilon$.
\end{proof}

As a consequence of Theorem~\ref{thm:flagged_decomposition}, we have $M_{\lambda}^{\mu} = \abs{\mcF_{\mu/\lambda}}$.
Furthermore, these are the conjugate of the flagged increasing tableaux of Lenart~\cite{Lenart00}, and so
\[
\G_{\lambda}(\xx; \beta) = \sum_{\mu \supseteq \lambda} M_{\lambda}^{\mu'} s_{\mu} = \sum_{\mu' \supseteq \lambda} M_{\lambda}^{\mu'} \omega s_{\mu'} = \omega \wG_{\lambda}(\xx; \beta),
\]
yielding a crystal-theoretic proof of Proposition~\ref{prop:conjugate_defn} (recall that $\omega$ is an involution).

We can construct the recording tableau we use to perform $\RSK^{-1}$ in the proof of Theorem~\ref{thm:flagged_decomposition} recursively following the description in constructing $\Upsilon^{-1}(b_1 \times F_1)$. We can clearly reconstruct $F_i$ by using the entries that are maximal in each column. For the $i$-th step, suppose the $i$-th column has height $h$, we increase all of the current entries by $h+1$, then we set the rightmost unset entry in row $j$ to $j+1$. Finally, for every entry in $F_i$ that is maximal in its column in row $j$, we set an entry to be $1$ in row $j$. We repeat this until we obtain a semistandard Young tableau.

\begin{ex}
Consider $b_1 \times F_1$ from Example~\ref{ex:uncrowding_big}. Then we construct the corresponding recording tableau $Q$ as
\begin{gather*}
\ytableaushort{{\cdot}{\cdot}{\cdot}{\cdot}{\cdot}{\cdot}112,{\cdot}{\cdot}{\cdot}13,114,5}\,,
\qquad
\ytableaushort{{\cdot}{\cdot}{\cdot}112445,{\cdot}1346,447,8}\,,
\allowdisplaybreaks\\
Q = \ytableaushort{112445778,34679,77{10},{11}}\,.
\end{gather*}
Performing the Lusztig involution on $b_1$ (in $\fsl_{10}$), we obtain
\[
T' = \ytableaushort{145667799,25688,388,4}\,,
\]
and then $\RSK^{-1}(T',Q)$ gives the word $w^* = 4583\;66886\;247899751$. When we apply the Lusztig involution to $w^*$ and obtain $w = 953112368\;42244\;7256$, which is precisely the reading word of $T$, and we can separate into columns based on the values of $Q$.
\end{ex}

Next, we construct a bijection with the set from Corollary~\ref{cor:hw_mst_skew_ssyt}.

\begin{prop}
Let $\overline{\ssyt}^{\lambda_1}(\mu/\lambda)$ be the set of semistandard Young tableaux from Corollary~\ref{cor:hw_mst_skew_ssyt}.
Then the map $\phi \colon \mcF_{\mu/\lambda} \to \overline{\ssyt}^{\lambda_1}(\mu/\lambda)$, where $\phi(T)$ is constructed by subtracting $i-\lambda_1$ to each entry of the $i$-th column of $T$, is a bijection.
\end{prop}

\begin{proof}
It is easy to see that the strictly increasing rows condition in $\mcF_{\mu/\lambda}$ is equivalent under $\phi$ to the weakly increasing rows condition in $\ssyt^k(\mu/\lambda)$. Furthermore, the maximum entry in column $i$ being $i$ condition in $\mcF_{\mu/\lambda}$ is equivalent under $\phi$ to the largest entry in column $i$ under $\phi$ in $\overline{\ssyt}^{\lambda_1}(\mu/\lambda)$ is $i - (i-\lambda_1) = \lambda_1$ (and hence for every row). Similarly the minimum entry in row $j$ for an increasing tableau in $\mcF_{\mu/\lambda}$ is $1$ in column $\lambda_j$, which is equivalent to the minimum entry for row $j$ in $\overline{\ssyt}^{\lambda_1}(\mu/\lambda)$ is $1 - (\lambda_j - \lambda_1) = \lambda_1 + 1 - \lambda_j$.
Thus $\phi$ is well-defined and surjective.
The map $\phi$ is clearly invertible, and hence $\phi$ is a bijection.
\end{proof}

\subsection{Weak Stable Grothendieck Functions}
\label{sec:weak_stable}

The \defn{$0$-Hecke monoid} is the monoid of all finite words in the alphabet $\{1,2,\dotsc,n\}$ subject to the relations
\begin{itemize}
\item $ij \equiv ji$ if $|i-j|>1$,
\item $iji \equiv jij$ if $|j-i|=1$,
\item $ii \equiv i$.
\end{itemize}
For any $w \in \sym_n$, let $\HH_w^k$ denote the set of words of length $k$ that are equivalent to some reduced expression for $w$ in the $0$-Hecke monoid (\textit{i.e} $w = s_{i_1} \dotsm s_{i_{\ell}}$ is considered as $i_1 \dotsm i_{\ell}$).
Note that this does not depend on the choice of reduced expression for $w$ by Matsumoto's theorem~\cite{Matsumoto64} (\textit{i.e.}, that any two reduced expressions for $w$ are related by the braid relations).

Next, let $\widehat{\HH}_{w,m}^k$ denote the set of two-line arrays
\[
\left[
\begin{array}{cccccccccccc}
1 & \cdots & 1 & 1 & 2 & \cdots & 2 & 2 & \cdots & m & \cdots & m \\
a_{1\ell_1} & \cdots & a_{12} & a_{11} & a_{2\ell_2} & \cdots & a_{22} & a_{21} & \cdots & a_{m\ell_m} & \cdots & a_{m1}
\end{array}
\right]
\]
such that $1 \leq a_{p1} \leq a_{p2} \leq \cdots \leq a_{p\ell_p} < n$ for all $1 \leq p \leq m$ (with possibly $\ell_p = 0$),
\begin{equation}
\label{eq:decreasing_factorization}
(a_{1\ell_1} \cdots a_{11}) (a_{2\ell_2} \cdots a_{21}) \cdots (a_{m\ell_m} \cdots a_{m1}) \equiv w,
\end{equation}
and $\sum_{p=1}^m \ell_p = k$.
Note that $\widehat{\HH}_{w,m}^k$ is equivalent to the ways of factorizing $w' \in \HH_w^k$ into $m$ weakly decreasing (possibly emtpy) factors as in Equation~\eqref{eq:decreasing_factorization}.
Let $\overline{\HH}_w^k$ denote the subset of $\widehat{\HH}_{w,m}^k$ such that $1 \leq a_{p1} < a_{p2} <\cdots < a_{p\ell_p} < n$ for all $1 \leq p \leq m$.
The strictly increasing condition $a_{p1} < \cdots < a_{p\ell_p}$ is equivalent to the notion of a compatible pair of words from~\cite{BKSTY08}.

Let $\mcP_w(\lambda)$ denote the set of increasing tableaux of shape $\lambda$ such that reading the entries of $P$ from top-to-bottom, right-to-left (\textit{i.e.}, also known as the Far-Eastern reading word) is equivalent to $w$ in the $0$-Hecke monoid.
Let $\svt(\lambda)_k$ (resp.~$\mvt(\lambda)_k$) denote the set of set-valued (resp. multiset-valued) tableaux $T$ such that $\abs{\wt{T}} = k$.
The \defn{(column) Hecke insertion} defined in~\cite{BKSTY08} is a bijection between $\overline{\HH}_{w,m}^k$ and $\bigsqcup_{\lambda} \mcP_w(\lambda) \times \svt(\lambda)_k$ (\cite[Lemma~1,Thm.~4]{BKSTY08}).\footnote{We note that the row Hecke insertion given in, \textit{e.g.},~\cite{PP16,sage,TY11} would yield a recording tableau that is the conjugate of a semistandard set-valued tableaux (\textit{i.e.}, rows would strictly increase and columns would weakly increase).}
Furthermore, we obtain the following from~\cite[Lemma~2]{BKSTY08}.

\begin{prop}
Hecke insertion defines a bijection between
\[
\widehat{\HH}_w^k \to \bigsqcup_{\lambda} \mcP_w(\lambda) \times \mvt(\lambda)_k.
\]
\end{prop}

\begin{dfn}
The \defn{weak stable Grothendieck polynomial} is defined to be
\begin{equation*}
\wG_w(\xx; \alpha) := \sum_{k=\ell(w)}^{\infty} \alpha^{k-\ell(w)} \sum_{(w,a) \in \widehat{\HH}_w^k} \prod_{i=1}^k x_{a_i},
\end{equation*}
where $a = a_1 a_2 \dotsm a_k \in \HH_w^k$.
\end{dfn}

The discussion above immediately implies that
\begin{equation*}
\wG_w(\xx; \alpha) = \sum_{k=\ell(w)}^{\infty} \alpha^{k-\ell(w)} \sum_{(P,\widehat{Q})} \wt(\widehat{Q}),
\end{equation*}
where we are summing over all $(P,\widehat{Q}) \in \bigsqcup_{\lambda} \mcP_w(\lambda) \times \mvt(\lambda)_k$. 
Putting this all together we obtain the following.

\begin{prop}
For any $w \in \sym_n$, we have
\begin{equation*}
\wG_w(\xx; \alpha) = \sum_{\lambda} \sum_{P \in \mcP_w(\lambda)} \sum_{k=\ell(w)}^{\infty} \alpha^{k-\ell(w)} \sum_{\widehat{Q}} s_{\wt(\widehat{Q})},
\end{equation*}
where we take the sum over all $\widehat{Q} \in \mvt(\lambda)_k$ such that $\widehat{Q}$ is a highest weight element.
Moreover, we have
\[
\wG_w(\xx; \alpha) = \sum_{\lambda} \alpha^{\abs{\lambda}-\ell(w)} \abs{\mcP_w(\lambda)} \wG_{\lambda}(\xx; \alpha).
\]
\end{prop}

\section{Crystal structure on hook-valued tableaux}
\label{sec:hooks}

Now we prove our second main result, that the hook-valued tableaux of~\cite{Yel17} admits a $U_q(\fsl_n)$-crystal structure that is isomorphic to a direct sum of highest weight crystals.
We do so by constructing a common generalization of the crystal structures on $\svt^n(\lambda)$ and $\mvt^n(\lambda)$.
Thus, we define the reading word using the reading words on $\svt^n(\lambda)$ and $\mvt^n(\lambda)$.

\begin{dfn}[Reading word]
Let $T \in \hvt^n(\lambda)$.
Let $C$ be a column of $T$. Define the \defn{column reading word} $\rd(C)$ by first reading the extended leg from largest to smallest in each box from bottom-to-top in $C$, then reading the entries in the arm from smallest to largest in each box from top-to-bottom in $C$.
Define the \defn{reading word} $\rd(T) = \rd(C_1) \rd(C_2) \dotsm \rd(C_k)$, where $C_1, C_2, \dotsc, C_k$ are the columns of $T$ from left-to-right.
\end{dfn}

\begin{ex}
For the hook-valued tableau
\[
\ytableausetup{boxsize=2.6em}
T = \; \ytableaushort{{
\begin{array}{@{}l@{}} \ml{1}1 \\[-4pt] \ml{3} \end{array}
}{
\begin{array}{@{}l@{}} \ml{4} \\[-4pt] \ml{5} \end{array}
},{
\begin{array}{@{}l@{}} \ml{4}47 \\[-4pt] \ml{5} \\[-4pt] \ml{6} \end{array}
}{
\begin{array}{@{}l@{}} \ml{7}779 \end{array}
},{
\begin{array}{@{}l@{}} \ml{8}99 \\[-4pt] \ml{9} \end{array}
}}\,,
\]
we have
\[
\rd(T) = \mathbf{\color{darkred}9865431} 1 47 99 \mathbf{\color{darkred} 754} 779.
\]
\end{ex}

Similarly, we define crystal operators by combining the set-valued crystal operators and the multiset-valued crystal operators.

\begin{dfn}[Crystal operators]
\label{defn:hvt_crystal_ops}
Fix $T \in \hvt^n(\lambda)$ and $i \in I$.
Write $\bplus$ for each $i$ in $\rd(T)$ and $\bminus$ for each $i+1$ in $\rd(T)$ (ignore all other letters). 
Next, cancel signs in ordered pairs $\bminus \bplus$ until obtaining a sequence of the form $\bplus \cdots \bplus \bminus \cdots \bminus$ called the \defn{$i$-signature}.
\begin{description}
\item[\defn{$e_i T$}] If there is no $\bminus$ in the resulting sequence, then $e_i T = 0$. Otherwise let $\bbb$ correspond to the box of the leftmost uncanceled $\bminus$. Then $e_i T$ is given by one of the following:
\begin{itemize}
\item[(M)] if there exists a box $\bbb^{\uparrow}$ immediately above $\bbb$ that contains an $i$, then remove an $i + 1$ from $A(\bbb)$ and add $i$ to $A(\bbb^{\uparrow})$;
\item[(S)] otherwise if there exists a box $\bbb^{\leftarrow}$ immediately to the left of $\bbb$ that contains an $i+1$ in $L(\bbb^{\leftarrow})$, then remove that $i+1$ from $L(\bbb^{\leftarrow})$ and add an $i$ to $L^+(\bbb)$;
\item[(N)] otherwise replace the $i+1$ in $\bbb$ with an $i$.
\end{itemize}

\item[\defn{$f_i T$}] If there is no $\bplus$ in the resulting sequence, then $f_i T = 0$. Otherwise let $\bbb$ correspond to the box of the rightmost uncanceled $\bplus$. Then $f_i T$ is given by one of the following:
\begin{itemize}
\item[(M)] if there exists a box $\bbb^{\downarrow}$ immediately below $\bbb$ that contains an $i+1$, then remove the $i$ from $A(\bbb)$ and add an $i+1$ to $A(\bbb^{\downarrow})$;
\item[(S)] otherwise if there exists a box $\bbb^{\rightarrow}$ immediately to the right of $\bbb$ that contains an $i$ in $L^+(\bbb^{\rightarrow})$, then remove the $i$ from $L^+(\bbb^{\rightarrow})$ and add an $i+1$ to $L(\bbb)$;
\item[(N)] otherwise replace the $i$ in $\bbb$ with an $i+1$.
\end{itemize}
\end{description}
\end{dfn}

Since the hook element of $\bbb^{\rightarrow}$ is $\min L^+(\bbb^{\rightarrow})$, if the element $i$ we remove from $L^+(\bbb^{\rightarrow})$ happens to also be the hook element, we can unambiguously give the new hook element of $\bbb^{\rightarrow}$ as $\min L(\bbb^{\rightarrow})$.
Similarly, if the added element $i < \min L^+(\bbb)$, then $i$ becomes the new hook element in $\bbb$.
Furthermore, we note that Case~(M) (resp.~(S)) correspond to the multiset-valued (resp.\ set-valued) tableaux crystal operators.
In particular, we are in Case~(S) when the $i+1$ (resp.~$i$) we are acting upon in $\bbb$ by $e_i$ (resp.~$f_i$) is in $L^+(\bbb)$.

\begin{ex}
The following connected components in $\hvt^3(2\Lambda_1)$ are those that correspond to $\alpha\beta$ and both are isomorphic to $B(\Lambda_2 + 2\Lambda_1)$:
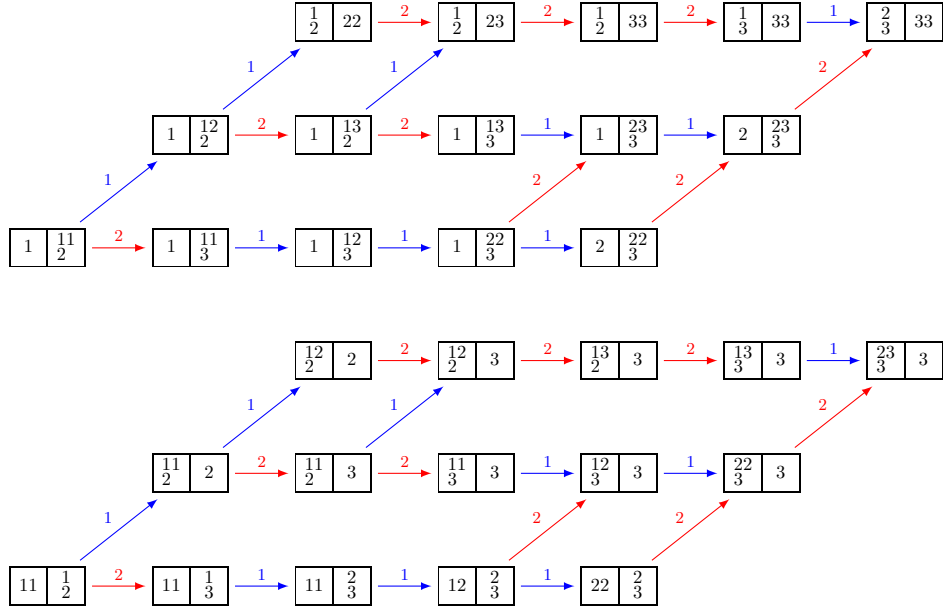
\[
\ytableausetup{boxsize=2.0em}
\begin{tikzpicture}[>=latex,xscale=1.9,yscale=1.5,every node/.style={scale=0.7}]
\node (hw) at (0,0) { $\ytableaushort{ 1{\begin{array}{@{}l@{}} 11 \\[-4pt] 2 \end{array}} }$ };
\node (f1) at (1,1) { $\ytableaushort{ 1{\begin{array}{@{}l@{}} 12 \\[-4pt] 2 \end{array}} }$ };
\node (f21) at (2,1) { $\ytableaushort{ 1{\begin{array}{@{}l@{}} 13 \\[-4pt] 2 \end{array}} }$ };
\node (f221) at (3,1) { $\ytableaushort{ 1{\begin{array}{@{}l@{}} 13 \\[-4pt] 3 \end{array}} }$ };
\node (f1221) at (4,1) { $\ytableaushort{ 1{\begin{array}{@{}l@{}} 23 \\[-4pt] 3 \end{array}} }$ };
\node (f11221) at (5,1) { $\ytableaushort{ 2{\begin{array}{@{}l@{}} 23 \\[-4pt] 3 \end{array}} }$ };
\node (f11) at (2,2) { $\ytableaushort{ {\begin{array}{@{}l@{}} 1 \\[-4pt] 2 \end{array}}{22} }$ };
\node (f211) at (3,2) { $\ytableaushort{ {\begin{array}{@{}l@{}} 1 \\[-4pt] 2 \end{array}}{23} }$ };
\node (f2211) at (4,2) { $\ytableaushort{ {\begin{array}{@{}l@{}} 1 \\[-4pt] 2 \end{array}}{33} }$ };
\node (f22211) at (5,2) { $\ytableaushort{ {\begin{array}{@{}l@{}} 1 \\[-4pt] 3 \end{array}}{33} }$ };
\node (f2) at (1,0) { $\ytableaushort{ 1{\begin{array}{@{}l@{}} 11 \\[-4pt] 3 \end{array}} }$ };
\node (f12) at (2,0) { $\ytableaushort{ 1{\begin{array}{@{}l@{}} 12 \\[-4pt] 3 \end{array}} }$ };
\node (f112) at (3,0) { $\ytableaushort{ 1{\begin{array}{@{}l@{}} 22 \\[-4pt] 3 \end{array}} }$ };
\node (f1112) at (4,0) { $\ytableaushort{ 2{\begin{array}{@{}l@{}} 22 \\[-4pt] 3 \end{array}} }$ };
\node (lw) at (6,2) { $\ytableaushort{ {\begin{array}{@{}l@{}} 2 \\[-4pt] 3 \end{array}}{33} }$ };
\draw[->,blue] (hw) -- node[midway,above left] {\small 1} (f1);
\draw[->,blue] (f1) -- node[midway,above left] {\small 1} (f11);
\draw[->,red] (hw) -- node[midway,above] {\small 2} (f2);
\draw[->,red] (f1) -- node[midway,above] {\small 2} (f21);
\draw[->,red] (f21) -- node[midway,above] {\small 2} (f221);
\draw[->,blue] (f21) -- node[midway,above left] {\small 1} (f211);
\draw[->,blue] (f221) -- node[midway,above] {\small 1} (f1221);
\draw[->,blue] (f1221) -- node[midway,above] {\small 1} (f11221);
\draw[->,red] (f11) -- node[midway,above] {\small 2} (f211);
\draw[->,red] (f211) -- node[midway,above] {\small 2} (f2211);
\draw[->,red] (f2211) -- node[midway,above] {\small 2} (f22211);
\draw[->,blue] (f22211) -- node[midway,above] {\small 1} (lw);
\draw[->,blue] (f2) -- node[midway,above] {\small 1} (f12);
\draw[->,blue] (f12) -- node[midway,above] {\small 1} (f112);
\draw[->,blue] (f112) -- node[midway,above] {\small 1} (f1112);
\draw[->,red] (f112) -- node[midway,above left] {\small 2} (f1221);
\draw[->,red] (f1112) -- node[midway,above left] {\small 2} (f11221);
\draw[->,red] (f11221) -- node[midway,above left] {\small 2} (lw);
\begin{scope}[yshift=-3cm]
\node (hw) at (0,0) { $\ytableaushort{ {11}{\begin{array}{@{}l@{}} 1 \\[-4pt] 2 \end{array}} }$ };
\node (f1) at (1,1) { $\ytableaushort{ {\begin{array}{@{}l@{}} 11 \\[-4pt] 2 \end{array}}2 }$ };
\node (f21) at (2,1) { $\ytableaushort{ {\begin{array}{@{}l@{}} 11 \\[-4pt] 2 \end{array}}3 }$ };
\node (f221) at (3,1) { $\ytableaushort{ {\begin{array}{@{}l@{}} 11 \\[-4pt] 3 \end{array}}3 }$ };
\node (f1221) at (4,1) { $\ytableaushort{ {\begin{array}{@{}l@{}} 12 \\[-4pt] 3 \end{array}}3 }$ };
\node (f11221) at (5,1) { $\ytableaushort{ {\begin{array}{@{}l@{}} 22 \\[-4pt] 3 \end{array}}3 }$ };
\node (f11) at (2,2) { $\ytableaushort{ {\begin{array}{@{}l@{}} 12 \\[-4pt] 2 \end{array}}2 }$ };
\node (f211) at (3,2) { $\ytableaushort{ {\begin{array}{@{}l@{}} 12 \\[-4pt] 2 \end{array}}3 }$ };
\node (f2211) at (4,2) { $\ytableaushort{ {\begin{array}{@{}l@{}} 13 \\[-4pt] 2 \end{array}}3 }$ };
\node (f22211) at (5,2) { $\ytableaushort{ {\begin{array}{@{}l@{}} 13 \\[-4pt] 3 \end{array}}3 }$ };
\node (f2) at (1,0) { $\ytableaushort{ {11}{\begin{array}{@{}l@{}} 1 \\[-4pt] 3 \end{array}} }$ };
\node (f12) at (2,0) { $\ytableaushort{ {11}{\begin{array}{@{}l@{}} 2 \\[-4pt] 3 \end{array}} }$ };
\node (f112) at (3,0) { $\ytableaushort{ {12}{\begin{array}{@{}l@{}} 2 \\[-4pt] 3 \end{array}} }$ };
\node (f1112) at (4,0) { $\ytableaushort{ {22}{\begin{array}{@{}l@{}} 2 \\[-4pt] 3 \end{array}} }$ };
\node (lw) at (6,2) { $\ytableaushort{ {\begin{array}{@{}l@{}} 23 \\[-4pt] 3 \end{array}}3 }$ };
\draw[->,blue] (hw) -- node[midway,above left] {\small 1} (f1);
\draw[->,blue] (f1) -- node[midway,above left] {\small 1} (f11);
\draw[->,red] (hw) -- node[midway,above] {\small 2} (f2);
\draw[->,red] (f1) -- node[midway,above] {\small 2} (f21);
\draw[->,red] (f21) -- node[midway,above] {\small 2} (f221);
\draw[->,blue] (f21) -- node[midway,above left] {\small 1} (f211);
\draw[->,blue] (f221) -- node[midway,above] {\small 1} (f1221);
\draw[->,blue] (f1221) -- node[midway,above] {\small 1} (f11221);
\draw[->,red] (f11) -- node[midway,above] {\small 2} (f211);
\draw[->,red] (f211) -- node[midway,above] {\small 2} (f2211);
\draw[->,red] (f2211) -- node[midway,above] {\small 2} (f22211);
\draw[->,blue] (f22211) -- node[midway,above] {\small 1} (lw);
\draw[->,blue] (f2) -- node[midway,above] {\small 1} (f12);
\draw[->,blue] (f12) -- node[midway,above] {\small 1} (f112);
\draw[->,blue] (f112) -- node[midway,above] {\small 1} (f1112);
\draw[->,red] (f112) -- node[midway,above left] {\small 2} (f1221);
\draw[->,red] (f1112) -- node[midway,above left] {\small 2} (f11221);
\draw[->,red] (f11221) -- node[midway,above left] {\small 2} (lw);
\end{scope}
\end{tikzpicture}
\]
Note that \scalebox{.7}{$\ytableaushort{ {\begin{array}{@{}l@{}} 11 \\[-4pt] 2 \end{array}}1 }$} is not semistandard.
\end{ex}

\begin{lemma}
\label{lemma:hvt_ops_well_defined}
Let $T, T' \in \hvt^n(\lambda)$. Then
\[
e_i T \in \hvt^n(\lambda) \sqcup \{0\},
\qquad
f_i T' \in \hvt^n(\lambda) \sqcup \{0\},
\qquad
e_i T = T' \Longleftrightarrow T = f_i T'.
\]
\end{lemma}

\begin{proof}
We can assume $e_i T \neq 0$ and $f_i T' \neq 0$ as the claim is trivial in these cases.
When we are in Case~(N) or Case(M) for the crystal operators, the proof of this is similar to the proof of Lemma~\ref{lemma:mvt_ops_well_defined}.
Thus, we assume the crystal operators are in Case~(S), and we assume the crystal operator acts on the box $\bbb$. Denote the following boxes around $\bbb$:
\[
\ytableaushort{{\bbb}{\bbb^{\rightarrow}},{\bbb^{\downarrow}}{\bbb^{\searrow}}}\,.
\]

Consider $f_i T'$. Since there exists an $i \in \bbb^{\rightarrow}$, there cannot exist an $i+1 \in \bbb$ by semistandardness. There also cannot be an $i+1 \in \bbb^{\downarrow}$ (if it exists) as otherwise we would be in Case~(M), and by semistandardness $i+1 \notin \bbb^{\searrow}$ (if it exists). We note that $i \in \bbb^{\rightarrow}$ must be canceled as otherwise we would be acting on $\bbb^{\rightarrow}$ by the signature rule. Hence, there must also exist an $i+1 \in L(\bbb^{\rightarrow})$. Thus, $f_i T'$ is defined and in $\hvt^n(\lambda)$.

From the above argument, we have that the $i+1 \in \bbb$ in $T = f_i T'$, which now cancels with the $i \in \bbb$. Thus the $i+1 \in L^+(\bbb^{\rightarrow})$ is now an unpaired $\bminus$, and by the semistandardness of $T'$, there does not exist an $i$ in the box immediately above $\bbb^{\rightarrow}$. Hence, we are in Case~(S) for $e_i T$ and clearly $e_i T = T'$.

The final claim $e_i T \in \hvt^n(\lambda)$ follows from the other two statements.
\end{proof}

Since either we have a pairing $\bminus \bplus$ or that $\varepsilon_i(T)$ (resp.~$\varphi_i(T)$) corresponds to the number of unpaired $\bminus$ (resp.~$\bplus$), which corresponds to an $i+1$ (resp.~$i$), we have
\[
\langle h_i, \wt(T) \rangle := m_i - m_{i+1} = \varphi_i(T) - \varepsilon_i(T),
\]
recall that $m_i$ is the number of $i$'s that appear in $T$.

\begin{thm}
\label{thm:hvt_crystal_structure}
Let $\lambda$ be a partition.
For any $T \in \hvt^n(\lambda)$ such that $T$ is a highest weight element, the closure of $T$ under the crystal operators is isomorphic to $B(\mu)$, where $\mu = \wt(T)$.
Moreover, we have
\[
\hvt^n(\lambda) \iso \bigoplus_{\mu \supseteq \lambda} B(\mu)^{\oplus H_{\lambda}^{\mu}},
\]
where $H_{\lambda}^{\mu}$ is the number of highest weight elements of weight $\mu$ in $\hvt^n(\lambda)$, and
\[
\hG_{\lambda}(\xx; \alpha, \beta) = \sum_{T} \alpha^{\sum_{\bbb \in T} \abs{A(\bbb)}} \beta^{\sum_{\bbb \in T} \abs{L(\bbb)}} s_{\wt(T)}(\xx),
\]
where the sum is taken over all highest weight elements in $\hvt^n(\lambda)$.
\end{thm}

\begin{proof}
Let $T \in \hvt^n(\lambda)$.

Clearly we have $\wt(f_i T) = \wt(T) \cdot \alpha_i^{-1}$, where $\alpha_i = x_i x_i^{-1}$ (this is the multiplicative version of one of the crystal axioms).
Thus, (P1) is satisfied by weight considerations and that $\hvt^n(\lambda)$ is a finite set and closed under the crystal operators (Lemma~\ref{lemma:mvt_ops_well_defined}).
(P2) is satisfied by Lemma~\ref{lemma:mvt_ops_well_defined}.

If $\abs{i - j} > 1$ (so $A_{ij} = 0$), then by the signature rule, we have $\Delta_i \delta_j(T) = \Delta_i \varphi_j(T) = 0$.
Now assume $\abs{i-j} = 1$. Then we have introduced an extra $i$ in the $j$-signature of $e_i T$. So by the signature rule, we have $\bigl( \Delta_i \delta_j(T), \Delta_i \varphi_j(T) \bigr) = (-1, 0), (0, -1)$ depending on if $j = i \pm 1$ and if the $i$ cancels in the $j$-signature or not. Hence, (P3) and (P4) are satisfied.


To show (P5), if $\abs{i - j} > 1$, then the signature rule implies (P5) since the $j$-signatures of $T$ and $e_i T$ are equal and similarly for the $i$-signatures of $T$ and $e_j T$. Furthermore, it is clear that $y := e_j e_i T = e_i e_j T$ with $\nabla_j \varphi_i(y) = 0$.

Therefore, assume $j = i \pm 1$ and $\Delta_i \delta_j(T) = 0$, that is to say $\varepsilon_j(e_i T) = \varepsilon_j(T)$. Hence, the $j$-signature of $e_i T$ is formed by either removing an uncanceled $\bplus$\footnote{The $e_i$ might act by removing a canceling $\bplus$ further to the left of the uncanceled $\bplus$ that is ultimately removed.} (if $j = i+1$) or adding a canceling $\bminus$ (if $j = i-1$) to the $j$-signature of $T$ at position $p'$. Note that only Case~(S) for $e_k$ moves a letter $k+1$ in position $\widetilde{p}$ in the reading word (as Lemma~\ref{lemma:fixed_reading} naturally extends to this setting) immediately to the right of the first $k+1$ to the right of $\widetilde{p}$, where it becomes a $k$. In any case, this does not affect the leftmost uncanceled $\bminus$ at position $p$, and hence $e_j$ acts on a $j+1$ in the same position $p$ in $\rd(T)$ and $\rd(e_i T)$. Note that we must have $p' < p$. Therefore, we form the $i$-signature of $e_j T$ from the $i$-signature of $T$ by adding an uncanceled $\bminus$ to the right of $p'$, which contributes the leftmost uncanceled $\bminus$ (which remains the leftmost uncanceled $\bminus$). Hence, $e_i$ acts on the same $i$ in $T$ and $e_j T$.

Next, suppose $e_i$ acts on box $\bbb$ in $T$. Note that $e_{i+1}$ cannot add an $i+1$ to the box $\bbb^{\leftarrow}$ immediately to the left of $\bbb$ unless there was already an $i+1 \in \bbb^{\leftarrow}$. Since $e_i$ acts on the same $i+1$ in $T$ and $e_j T$, the $e_{i-1}$ cannot move the only $i$ from the box $\bbb^{\uparrow}$ immediately above $\bbb$. Hence, $e_i$ acts using the same case in $T$ as in $e_j T$, and similarly for $e_j$ acting by the same case in $T$ and $e_i T$.
Therefore, we have $y := e_j e_i T = e_i e_j T$ with $\nabla_j \varphi_i(y) = 0$.

Now we show (P6), and so we assume $\Delta_i \delta_j(T) = \Delta_j \delta_i(T) = -1$. Due to the symmetry, we assume $j = i+1$ without loss of generality.
We must form the $i$-signature (resp.\ $(i+1)$-signature) of $e_{i+1} T$ (resp.~$e_i T$) from the $i$-signature (resp.\ $(i+1)$-signature) of $T$ by adding an uncanceled $\bminus$ (resp.\ removing a canceling $\bplus$).
As a result, if $e_i$ (resp.~$e_{i+1}$) acts on position $p$ (resp.~$p'$) in the reduced word, we must have $p < p'$ with an additional $i+1$ before position $p$.
Let $k = i+2$. By semistandardness, we have one of the following configurations on $i, j, k$ (where we ignore all other letters):
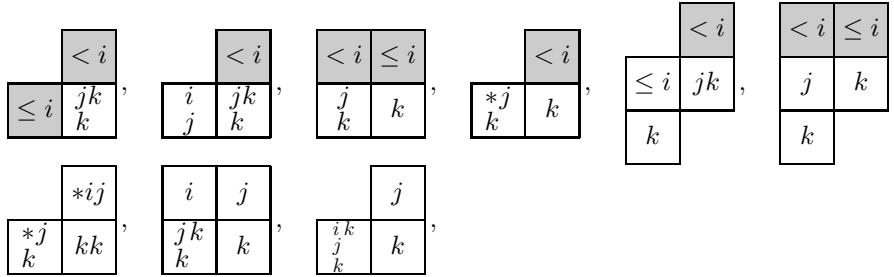
\[
\begin{array}{c@{\quad\;}c@{\quad\;}c@{\quad\;}c@{\quad\;}c@{\quad\;}c}
\ytableausetup{boxsize=2.0em}
\ytableaushort{ \none{*(black!20) < i},{*(black!20) \leq i}{\begin{array}{@{}l@{}} {j}{k} \\[-2pt] {k} \end{array}} }\,,
&
\ytableaushort{ \none{*(black!20) < i},{\begin{array}{c} i \\[-2pt] j \end{array}}{\begin{array}{@{}l@{}} jk \\[-2pt] k \end{array}} }\,,
&
\ytableaushort{ {*(black!20) < i}{*(black!20) \leq i},{\begin{array}{c} j \\[-2pt] k \end{array}}{k} }\,,
&
\ytableaushort{ \none{*(black!20)<i},{\begin{array}{@{}l@{}l@{}} \ast & j \\[-2pt] k \end{array}}k }\,,
&
\ytableaushort{ \none{*(black!20) < i},{\leq i}{jk},k }\,,
&
\ytableaushort{ {*(black!20) < i}{*(black!20) \leq i},jk,k }\,,
\\[25pt]
\ytableaushort{ \none{\ast ij},{\begin{array}{@{}l@{}l@{}} {\ast}&{j} \\[-2pt] {k} \end{array}}{kk} }\,,
&
\ytableaushort{ ij,{\begin{array}{@{}l@{}l@{}} j&k \\[-2pt] k \end{array}}{k} }\,,
&
\ytableaushort{{\none}j,{\scalebox{0.75}{$\begin{array}{@{}l@{}l@{}} i & k \\[-4pt] j \\[-1.5pt] k \end{array}$}}k}\,,
\end{array}
\]
and with the boxes not necessarily adjacent,
where $e_i$ (resp.~$e_j$) acts on the rightmost $j$ (resp.~$k$), the shaded boxes do not necessarily exist, and $\ast \leq i$ (with possibly more than one entry in the arm).
Note that if the boxes are not adjacent where a Case~(S) occurs in the above configurations, then we cannot have Case~(S) occur in any of the crystal operators, and so the result holds from Lemma~\ref{lemma:fixed_reading}, which holds in this setting for Case~(M) or Case~(N), and the crystal of words.
Similarly, if a Case~(S) occurs in a box that does not have one of the local configuration, then the result follows from the crystal of words and the boxes do not interact.
Thus, it is a finite check to see that we have $y := e_i e_j^2 e_i T = e_j e_i^2 e_j T$ and $\nabla_i \varphi_j(y) = \nabla_j \varphi_i(y) = -1$.

The proof of (P5') and (P6') is similar. Hence the Stembridge axioms hold. Note that the crystal operators preserve the sum of the arm lengths and the sum of the leg lengths. Therefore, the exponent of $x_i$ in the weight of any $T$ in a connected component for $\alpha^a \beta^b$ is bounded above by $a+b$, and so every connected component has a highest weight element. Therefore, we have
\[
\hvt^n(\lambda) \iso \bigoplus_{\mu \supseteq \lambda} B(\mu)^{\oplus H_{\lambda}^{\mu}},
\quad
\hG_{\lambda}(\xx; \alpha, \beta) = \sum_{T} \alpha^{\sum_{\bbb \in T} \abs{A(\bbb)}} \beta^{\sum_{\bbb \in T} \abs{L(\bbb)}} s_{\wt(T)}(\xx),
\]
by Theorem~\ref{thm:stembridge}.
\end{proof}

As an immediate consequence of Theorem~\ref{thm:hvt_crystal_structure}, we have that $\hG_{\lambda}(\xx; \alpha, \beta)$ is Schur positive.

\begin{ex}
We have the following local relation in $\hvt^3(2\Lambda_1)$ on the left and their corresponding reading words on the right:
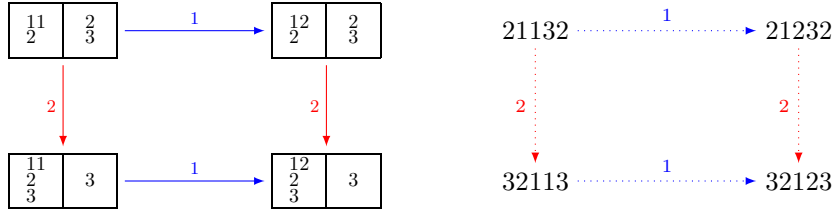
\[
\ytableausetup{boxsize=2.7em}
\begin{tikzpicture}[>=latex,xscale=3.5,yscale=2,every node/.style={scale=0.75},baseline=0]
\node (b) at (0,0) { $\ytableaushort{ {\begin{array}{@{}l@{}} 11 \\[-4pt] 2 \end{array}}{\begin{array}{@{}l@{}} 2 \\[-4pt] 3 \end{array}} }$ };
\node (f1) at (1,0) { $\ytableaushort{ {\begin{array}{@{}l@{}} 12 \\[-4pt] 2 \end{array}}{\begin{array}{@{}l@{}} 2 \\[-4pt] 3 \end{array}} }$ };
\node (f2) at (0,-1) { $\ytableaushort{ {\begin{array}{@{}l@{}} 11 \\[-4pt] 2 \\[-4pt] 3 \end{array}}{\begin{array}{@{}l@{}} 3 \end{array}} }$ };
\node (bp) at (1,-1) { $\ytableaushort{ {\begin{array}{@{}l@{}} 12 \\[-4pt] 2 \\[-4pt] 3  \end{array}}{\begin{array}{@{}l@{}} 3\end{array}} }$ };
\draw[->,blue] (b) -- node[midway,above] {\small 1} (f1);
\draw[->,red] (b) -- node[midway,left] {\small 2} (f2);
\draw[->,blue] (f2) -- node[midway,above] {\small 1} (bp);
\draw[->,red] (f1) -- node[midway,left] {\small 2} (bp);
\end{tikzpicture}
\qquad\qquad
\begin{tikzpicture}[>=latex,xscale=3.5,yscale=2,baseline=0]
\node (b) at (0,0) { $21132$ };
\node (f1) at (1,0) { $21232$ };
\node (f2) at (0,-1) { $32113$ };
\node (bp) at (1,-1) { $32123$ };
\draw[->,dotted,blue] (b) -- node[midway,above] {\scriptsize 1} (f1);
\draw[->,dotted,red] (b) -- node[midway,left] {\scriptsize 2} (f2);
\draw[->,dotted,blue] (f2) -- node[midway,above] {\scriptsize 1} (bp);
\draw[->,dotted,red] (f1) -- node[midway,left] {\scriptsize 2} (bp);
\end{tikzpicture}
\]
Let $T$ be the upper-left hook-valued tableau.
We note that while the position of the $1$ that is acted on by $f_1$ in $T$ and $f_2 T$, it is still the second $1$ in the reading word. Thus $\varphi_1(T) = \varphi_1(f_2 T)$, and we have $f_1 f_2 T = f_2 f_1 T$.
\end{ex}

\section{Crystal structure on valued-set tableaux}
\label{sec:vst}

In this section, we define a $U_q(\fsl_n)$-crystal structure on valued-set tableaux that is isomorphic to a direct sum of highest weight crystals. Furthermore, we give an explicit crystal isomorphism with the usual crystal structure on semistandard Young tableaux through the inflation map.

\subsection{Crystal structure}

\begin{dfn}[Reading word]
Let $T \in \vst^n(\lambda)$. Define the \defn{reading word} $\rd(T)$ to be the reading word of the usual reverse Far-Eastern (bottom-to-top, left-to-right) reading word of the tableau where we only consider the buoy entries.
\end{dfn}

\begin{ex}
\label{ex:buoy_reading_word}
Consider the valued-set tableau
\[
T = 
\begin{tikzpicture}[scale=0.5,baseline=-30]
\foreach \i/\r in {0/10,1/10,2/9,3/5,4/3,5/2}
    \draw (0,-\i) -- (\r,-\i);
\draw (0,0) -- (0,-5);
\draw (1,-2) -- (1,-3);
\draw (2,0) -- (2,-1);
\draw (2,-2) -- (2,-5);
\draw (3,-1) -- (3,-4);
\draw (5,-1) -- (5,-3);
\draw (6,0) -- (6,-2);
\draw (7,0) -- (7,-1);
\draw (8,-1) -- (8,-2);
\draw (9,0) -- (9,-2);
\draw (10,0) -- (10,-1);
\draw (0.5,-0.5) node {$\mathbf{\color{darkred}1}$};
\draw (2.5,-0.5) node {$\mathbf{\color{darkred}1}$};
\foreach \i in {1,3,4,5}
    \draw (\i+0.5,-0.5) node[gray] {$1$};
\draw (6.5,-0.5) node {$\mathbf{\color{darkred}2}$};
\draw (0.5,-1.5) node {$\mathbf{\color{darkred}2}$};
\foreach \i in {1,2}
    \draw (\i+0.5,-1.5) node[gray] {$2$};
\draw (0.5,-2.5) node {$\mathbf{\color{darkred}3}$};
\draw (3.5,-1.5) node {$\mathbf{\color{darkred}3}$};
\draw (4.5,-1.5) node[gray] {$3$};
\draw (7.5,-0.5) node {$\mathbf{\color{darkred}4}$};
\draw (8.5,-0.5) node[gray] {$4$};
\draw (5.5,-1.5) node {$\mathbf{\color{darkred}4}$};
\draw (1.5,-2.5) node {$\mathbf{\color{darkred}4}$};
\draw (6.5,-1.5) node {$\mathbf{\color{darkred}5}$};
\draw (7.5,-1.5) node[gray] {$5$};
\draw (8.5,-1.5) node {$\mathbf{\color{darkred}5}$};
\draw (2.5,-2.5) node {$\mathbf{\color{darkred}5}$};
\draw (3.5,-2.5) node {$\mathbf{\color{darkred}5}$};
\draw (4.5,-2.5) node[gray] {$5$};
\draw (0.5,-3.5) node {$\mathbf{\color{darkred}5}$};
\draw (1.5,-3.5) node[gray] {$5$};
\draw (9.5,-0.5) node {$\mathbf{\color{darkred}6}$};
\draw (0.5,-4.5) node {$\mathbf{\color{darkred}7}$};
\draw (1.5,-4.5) node[gray] {$7$};
\draw (2.5,-3.5) node {$\mathbf{\color{darkred}8}$};
\end{tikzpicture}\,,
\]
where we have written the buoys in bold, and we have
\[
\rd(T) = 75321485153452456,
\qquad\qquad
\wt(T) = x_1^2 x_2^2 x_3^2 x_4^3 x_5^5 x_6 x_7 x_8.
\]
\end{ex}

\begin{dfn}[Crystal operators]
Fix $T \in \hvt^n(\lambda)$ and $i \in I$.
Write $\bplus$ for each $i$ in $\rd(T)$ and $\bminus$ for each $i+1$ in $\rd(T)$ (ignore all other letters). 
Next, cancel signs in ordered pairs $\bminus \bplus$ until obtaining a sequence of the form $\bplus \cdots \bplus \bminus \cdots \bminus$ called the \defn{$i$-signature}.

\begin{description}
\item[\defn{$e_i T$}] If there is no $\bminus$ in the resulting sequence, then $e_i T = 0$. Otherwise let $\gp$ correspond to the group of the leftmost uncanceled $\bminus$. Then $e_i T$ is given by one of the following:
\begin{itemize}
\item if the entry immediately above the anchor of $\gp$ is an $i$, move the divider between $\gp$ and the group immediately to the left $\gp$ one step up;
\item otherwise change every $i+1$ to an $i$ in $\gp$.
\end{itemize}

\item[\defn{$f_i T$}] If there is no $\bplus$ in the resulting sequence, then $f_i T = 0$. Otherwise let $\bbb$ correspond to the box of the rightmost uncanceled $\bplus$. Then $f_i T$ is given by one of the following:
\begin{itemize}
\item if the entry immediately below the anchor of $\gp$ in an $i+1$, move the divider between $\gp$ and the group immediately to the left $\gp$ one step down;
\item otherwise change every $i$ to an $i+1$ in $\gp$.
\end{itemize}
\end{description}
\end{dfn}

\begin{ex}
The following is a connected component in $\vst^3(3\Lambda_2)$ corresponding to $\alpha^2$ and is isomorphic to $B(\Lambda_2 + 2 \Lambda_1)$:
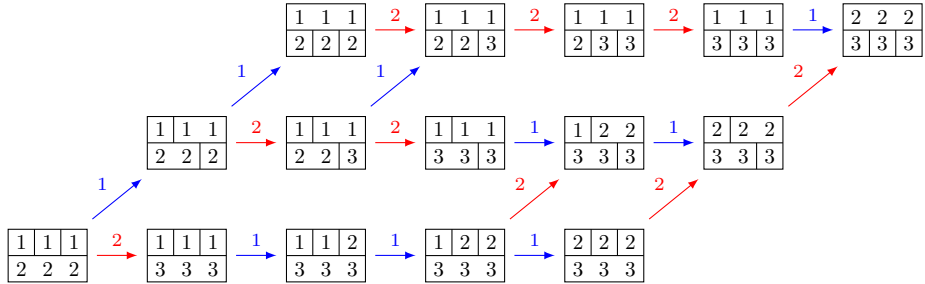
\[
\ytableausetup{boxsize=1.5em,tabloids}
\begin{tikzpicture}[>=latex,xscale=1.85,yscale=1.5]
\newcommand{\tab}[1]{\begin{tikzpicture}[scale=0.35,every node/.style={scale=0.8}] \draw (0,0) rectangle (3,-2); \draw (0,-1) -- (3,-1); #1 \end{tikzpicture}}
\node (hw) at (0,0) { $\tab{
  \foreach \x in {1,2,3} {
    \draw (\x-0.5, -0.5) node {$1$};
    \draw (\x-0.5, -1.5) node {$2$};
  }
  \draw(1,0) -- (1,-1);
  \draw(2,0) -- (2,-1);
}$ };
\node (f1) at (1,1) { $\tab{
  \foreach \x in {1,2,3} {
    \draw (\x-0.5, -0.5) node {$1$};
    \draw (\x-0.5, -1.5) node {$2$};
  }
  \draw(1,0) -- (1,-1);
  \draw(2,-1) -- (2,-2);
}$ };
\node (f21) at (2,1) { $\tab{
  \foreach \x in {1,2,3}
    \draw (\x-0.5, -0.5) node {$1$};
  \draw (1-0.5, -1.5) node {$2$};
  \draw (2-0.5, -1.5) node {$2$};
  \draw (3-0.5, -1.5) node {$3$};
  \draw(1,0) -- (1,-1);
  \draw(2,-1) -- (2,-2);
}$ };
\node (f221) at (3,1) { $\tab{
  \foreach \x in {1,2,3} {
    \draw (\x-0.5, -0.5) node {$1$};
    \draw (\x-0.5, -1.5) node {$3$};
  }
  \draw(1,0) -- (1,-1);
  \draw(2,-1) -- (2,-2);
}$ };
\node (f1221) at (4,1) { $\tab{
  \foreach \x in {1,2,3}
    \draw (\x-0.5, -1.5) node {$3$};
  \draw (1-0.5, -0.5) node {$1$};
  \draw (2-0.5, -0.5) node {$2$};
  \draw (3-0.5, -0.5) node {$2$};
  \draw(1,0) -- (1,-1);
  \draw(2,-1) -- (2,-2);
}$ };
\node (f11221) at (5,1) { $\tab{
  \foreach \x in {1,2,3} {
    \draw (\x-0.5, -0.5) node {$2$};
    \draw (\x-0.5, -1.5) node {$3$};
  }
  \draw(1,0) -- (1,-1);
  \draw(2,-1) -- (2,-2);
}$ };
\node (f11) at (2,2) { $\tab{
  \foreach \x in {1,2,3} {
    \draw (\x-0.5, -0.5) node {$1$};
    \draw (\x-0.5, -1.5) node {$2$};
  }
  \draw(1,-1) -- (1,-2);
  \draw(2,-1) -- (2,-2);
}$ };
\node (f211) at (3,2) { $\tab{
  \foreach \x in {1,2,3} {
    \draw (\x-0.5, -0.5) node {$1$};
  }
  \draw (1-0.5, -1.5) node {$2$};
  \draw (2-0.5, -1.5) node {$2$};
  \draw (3-0.5, -1.5) node {$3$};
  \draw(1,-1) -- (1,-2);
  \draw(2,-1) -- (2,-2);
}$ };
\node (f2211) at (4,2) { $\tab{
  \foreach \x in {1,2,3} {
    \draw (\x-0.5, -0.5) node {$1$};
  }
  \draw (1-0.5, -1.5) node {$2$};
  \draw (2-0.5, -1.5) node {$3$};
  \draw (3-0.5, -1.5) node {$3$};
  \draw(1,-1) -- (1,-2);
  \draw(2,-1) -- (2,-2);
}$ };
\node (f22211) at (5,2) { $\tab{
  \foreach \x in {1,2,3} {
    \draw (\x-0.5, -0.5) node {$1$};
    \draw (\x-0.5, -1.5) node {$3$};
  }
  \draw(1,-1) -- (1,-2);
  \draw(2,-1) -- (2,-2);
}$ };
\node (f2) at (1,0) { $\tab{
  \foreach \x in {1,2,3} {
    \draw (\x-0.5, -0.5) node {$1$};
    \draw (\x-0.5, -1.5) node {$3$};
  }
  \draw(1,0) -- (1,-1);
  \draw(2,0) -- (2,-1);
}$ };
\node (f12) at (2,0) { $\tab{
  \draw (1-0.5, -0.5) node {$1$};
  \draw (2-0.5, -0.5) node {$1$};
  \draw (3-0.5, -0.5) node {$2$};
  \foreach \x in {1,2,3} {
    \draw (\x-0.5, -1.5) node {$3$};
  }
  \draw(1,0) -- (1,-1);
  \draw(2,0) -- (2,-1);
}$ };
\node (f112) at (3,0) { $\tab{
  \draw (1-0.5, -0.5) node {$1$};
  \draw (2-0.5, -0.5) node {$2$};
  \draw (3-0.5, -0.5) node {$2$};
  \foreach \x in {1,2,3} {
    \draw (\x-0.5, -1.5) node {$3$};
  }
  \draw(1,0) -- (1,-1);
  \draw(2,0) -- (2,-1);
}$ };
\node (f1112) at (4,0) { $\tab{
  \foreach \x in {1,2,3} {
    \draw (\x-0.5, -0.5) node {$2$};
    \draw (\x-0.5, -1.5) node {$3$};
  }
  \draw(1,0) -- (1,-1);
  \draw(2,0) -- (2,-1);
}$ };
\node (lw) at (6,2) { $\tab{
  \foreach \x in {1,2,3} {
    \draw (\x-0.5, -0.5) node {$2$};
    \draw (\x-0.5, -1.5) node {$3$};
  }
  \draw(1,-1) -- (1,-2);
  \draw(2,-1) -- (2,-2);
}$ };
\draw[->,blue] (hw) -- node[midway,above left] {\scriptsize 1} (f1);
\draw[->,blue] (f1) -- node[midway,above left] {\scriptsize 1} (f11);
\draw[->,red] (hw) -- node[midway,above] {\scriptsize 2} (f2);
\draw[->,red] (f1) -- node[midway,above] {\scriptsize 2} (f21);
\draw[->,red] (f21) -- node[midway,above] {\scriptsize 2} (f221);
\draw[->,blue] (f21) -- node[midway,above left] {\scriptsize 1} (f211);
\draw[->,blue] (f221) -- node[midway,above] {\scriptsize 1} (f1221);
\draw[->,blue] (f1221) -- node[midway,above] {\scriptsize 1} (f11221);
\draw[->,red] (f11) -- node[midway,above] {\scriptsize 2} (f211);
\draw[->,red] (f211) -- node[midway,above] {\scriptsize 2} (f2211);
\draw[->,red] (f2211) -- node[midway,above] {\scriptsize 2} (f22211);
\draw[->,blue] (f22211) -- node[midway,above] {\scriptsize 1} (lw);
\draw[->,blue] (f2) -- node[midway,above] {\scriptsize 1} (f12);
\draw[->,blue] (f12) -- node[midway,above] {\scriptsize 1} (f112);
\draw[->,blue] (f112) -- node[midway,above] {\scriptsize 1} (f1112);
\draw[->,red] (f112) -- node[midway,above left] {\scriptsize 2} (f1221);
\draw[->,red] (f1112) -- node[midway,above left] {\scriptsize 2} (f11221);
\draw[->,red] (f11221) -- node[midway,above left] {\scriptsize 2} (lw);
\end{tikzpicture}
\ytableausetup{notabloids}
\]
and a connected component isomorphic to $B(2\Lambda_2)$:
\[
\newcommand{\tab}[1]{\begin{tikzpicture}[scale=0.35,every node/.style={scale=0.8}] \draw (0,0) rectangle (3,-2); \draw (0,-1) -- (3,-1); \draw (1,0) -- (1,-2); #1 \end{tikzpicture}}
\begin{tikzpicture}[>=latex,xscale=1.85,yscale=1.5]
\node (hw) at (0,0) { $\tab{
  \foreach \x in {2,3} {
    \draw (\x-0.5, -0.5) node {$1$};
    \draw (\x-0.5, -1.5) node {$2$};
  }
  \draw (0.5,-0.5) node {$1$};
  \draw (0.5,-1.5) node {$2$};
} $};
\node (f2) at (1,0) { $\tab{
  \foreach \x in {2,3} {
    \draw (\x-0.5, -0.5) node {$1$};
    \draw (\x-0.5, -1.5) node {$3$};
  }
  \draw (0.5,-0.5) node {$1$};
  \draw (0.5,-1.5) node {$2$};
} $};
\node (f22) at (2,0) { $\tab{
  \foreach \x in {2,3} {
    \draw (\x-0.5, -0.5) node {$1$};
    \draw (\x-0.5, -1.5) node {$3$};
  }
  \draw (0.5,-0.5) node {$1$};
  \draw (0.5,-1.5) node {$3$};
} $};
\node (f12) at (2,1) { $\tab{
  \foreach \x in {2,3} {
    \draw (\x-0.5, -0.5) node {$2$};
    \draw (\x-0.5, -1.5) node {$3$};
  }
  \draw (0.5,-0.5) node {$1$};
  \draw (0.5,-1.5) node {$2$};
} $};
\node (f122) at (3,1) { $\tab{
  \foreach \x in {2,3} {
    \draw (\x-0.5, -0.5) node {$2$};
    \draw (\x-0.5, -1.5) node {$3$};
  }
  \draw (0.5,-0.5) node {$1$};
  \draw (0.5,-1.5) node {$3$};
} $};
\node (f1122) at (4,1) { $\tab{
  \foreach \x in {2,3} {
    \draw (\x-0.5, -0.5) node {$2$};
    \draw (\x-0.5, -1.5) node {$3$};
  }
  \draw (0.5,-0.5) node {$2$};
  \draw (0.5,-1.5) node {$3$};
} $};
\draw[->,red] (hw) -- node[midway,above] {\scriptsize 2} (f2);
\draw[->,red] (f2) -- node[midway,above] {\scriptsize 2} (f22);
\draw[->,blue] (f2) -- node[midway,above left] {\scriptsize 1} (f12);
\draw[->,red] (f12) -- node[midway,above] {\scriptsize 2} (f122);
\draw[->,blue] (f22) -- node[midway,above left] {\scriptsize 1} (f122);
\draw[->,blue] (f122) -- node[midway,above] {\scriptsize 1} (f1122);
\end{tikzpicture}
\]
\end{ex}

\begin{lemma}
\label{lemma:crystal_structure_vst}
Let $T, T' \in \vst^n(\lambda)$. Then, we have
\[
e_i T \in \vst^n(\lambda) \sqcup \{0\},
\qquad
f_i T' \in \vst^n(\lambda) \sqcup \{0\},
\qquad
e_i T = T' \Longleftrightarrow T = f_i T',
\]
and $\rd$ is a strict crystal embedding of $\vst^n(\lambda)$
\end{lemma}

\begin{proof}
It is clear all claims hold when a divider does not move, and so we assume a divider moves under the crystal operators.
In particular, if the divider does move under $f_i$, then it means locally the crystal operator $f_i$ does the steps
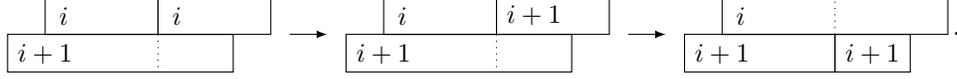
\[
\begin{tikzpicture}[scale=0.5,>=latex,baseline=0]
\draw[-] (0,0) rectangle (3,1);
\draw[-] (3,0) rectangle (6,1);
\draw[-] (-1,0) rectangle (5,-1);
\draw[dotted] (3,0) -- (3,-1);
\draw (0,0) node[anchor=north] {$i+1$};
\draw (0.5,0) node[anchor=south] {$i$};
\draw (3.5,0) node[anchor=south] {$i$};
\draw[->] (6.5,0) -- (7.5,0);
\begin{scope}[xshift=9cm]
\draw[-] (0,0) rectangle (3,1);
\draw[-] (3,0) rectangle (6,1);
\draw[-] (-1,0) rectangle (5,-1);
\draw[dotted] (3,0) -- (3,-1);
\draw (0,0) node[anchor=north] {$i+1$};
\draw (0.5,0) node[anchor=south] {$i$};
\draw (4,0) node[anchor=south] {$i+1$};
\end{scope}
\draw[->] (15.5,0) -- (16.5,0);
\begin{scope}[xshift=18cm]
\draw[-] (0,0) rectangle (6,1);
\draw[-] (-1,0) rectangle (3,-1);
\draw[-] (3,0) rectangle (5,-1);
\draw[dotted] (3,0) -- (3,1);
\draw (0,0) node[anchor=north] {$i+1$};
\draw (0.5,0) node[anchor=south] {$i$};
\draw (4,0) node[anchor=north] {$i+1$};
\end{scope}
\end{tikzpicture}\,.
\]
Note that by the semistandardness, the anchor for the $i+1$ group must be weakly to the left of the anchor for the left $i$ group, but the group for the $i+1$ may extend further to the right than the right $i$ group.
The operator $e_i$ does the above steps in reverse order.
Therefore, it is straightforward to see the claims follow by the definition of the reading word.
\end{proof}

\begin{thm}
\label{thm:vst_crystal}
Let $\lambda$ be a partition. We have
\[
\vst^n(\lambda) \iso \bigoplus_{\mu \subseteq \lambda} B(\mu)^{\oplus V_{\lambda}^{\mu}},
\]
where $V_{\lambda}^{\mu}$ is the number of highest weight elements of weight $\mu$ in $\vst^n(\lambda)$. Moreover, we have
\[
\dwG_{\lambda}(\xx; \alpha) = \sum_{\mu \subseteq \lambda} \alpha^{\abs{\lambda} - \abs{\mu}} V_{\lambda}^{\mu} s_{\mu}(\xx).
\]
\end{thm}

\begin{proof}
First note that the  numbers, $b_i$, of  buoy entries in column $i$ are preserved under all crystal operations and therefore are the same amongst all elements of a connected component of $VST^n(\lambda)$.  Moreover, it is not difficult to check that for any specified values of $b_i$ there is at most one valued-set tableau of shape $\lambda$ for any given reading word.  Thus, within a connected component the map $T \to \rd(T)$ is injective.  The result now follows from Lemma~\ref{lemma:crystal_structure_vst} and that every $B(\mu)$ is a subcrystal of $B(\Lambda_1)^{\otimes \abs{\mu}}$.
\end{proof}

We also have analogs of Proposition~\ref{prop:hw_mst}, Corollary~\ref{cor:same_length_mst}, and Corollary~\ref{cor:hw_mst_skew_ssyt}.

\begin{prop}
\label{prop:vst_hw_condition}
Suppose $T \in \vst^n(\lambda)$ is a highest weight element.  Then the $i$-th row of $T$ contains only instances of the letter $i$.
\end{prop}

\begin{proof}
This is similar to the proof of Proposition~\ref{prop:hw_mst} (or for $B(\lambda)$), for a highest weight element $T$.
\end{proof}

\begin{cor}
If $V_{\mu}^{\lambda} \neq 0$, then $\ell(\lambda) = \ell(\mu)$.
\end{cor}

\begin{proof}
This follows from Theorem~\ref{thm:vst_crystal}, Proposition~\ref{prop:vst_hw_condition}, and that every row of a valued-set tableau must have at least one group.
\end{proof}

Recall that a \defn{conjugate semistandard tableau} is a tableau that is weakly increasing down columns and strictly increasing across rows.

\begin{cor}
\label{cor:vst_hw_counting}
The number of highest weight valued-set tableaux of shape $\lambda = (\lambda_1, \lambda_2, \dotsc, \lambda_{\ell})$ and weight $\mu = (\mu_1, \mu_2, \dotsc, \mu_{\ell})$ is the number of conjugate semistandard tableaux of shape $(\mu_1-1, \mu_2-1, \dotsc, \mu_{\ell} - 1)$ such that all of the entries of the $i$-th row are strictly less than $\lambda_i$.
\end{cor}

\begin{proof}
This follows from Proposition~\ref{prop:vst_hw_condition}, which allows us to construct a bijection such that the value in box $\bbb$ in row $i$ and column $j$ of a such a conjugate semistandard tableau corresponds to the position of the buoy (or anchor) of the $j$-th group of row $i$ in a highest weight valued-set tableau with the last group in each row being fixed. Note that the column weakly increasing condition precisely corresponds to canceling pairs in the $i$-signature.
\end{proof}

We can also parameterize the highest weight elements analogous to Corollary~\ref{cor:hw_mst_skew_ssyt}.

\begin{cor}
\label{cor:vst_hw_rpp_formula}
The number of highest weight valued-set tableaux of shape $\lambda$ and weight $\mu$ is the number of reverse plane partitions of shape $\lambda / \mu$ such that all of the entries of the $i$-th row are in the (closed) interval $[\mu_1+1-\mu_i,\mu_1]$.
\end{cor}

\begin{proof}
This follows from Proposition~\ref{prop:vst_hw_condition}, which allows us to construct a bijection such that the $j$-th box of the $i$-th row of $\lambda / \mu$ is $N_{i,j} + \mu_1 - \mu_i$, where $N_{i,j}$ is the number of buoy entries to the left of the $j$-th nonbuoy entry of row $i$ of the highest weight valued-set tableau.
\end{proof}

\begin{ex}
The following are the highest weight elements of $\vst^3(3\Lambda_2)$ and their corresponding reverse plane partitions from Corollary~\ref{cor:vst_hw_rpp_formula}:
\newcommand{\tab}[1]
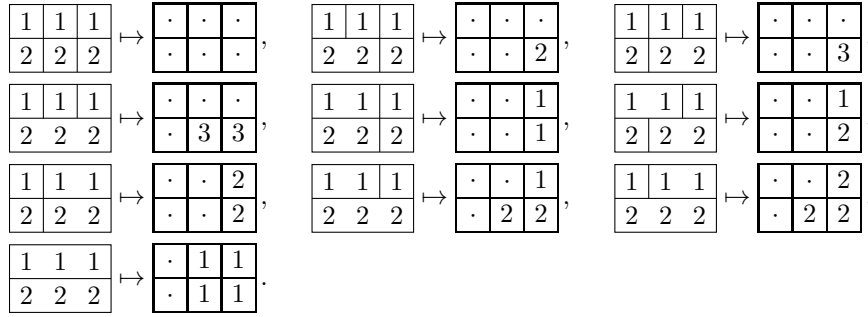
{
\begin{tikzpicture}[scale=0.45,baseline=-0.53cm]
\draw (0,0) rectangle (3,-2); \draw (0,-1) -- (3,-1);
\foreach \x in {1,2,3} { \draw (\x-0.5, -0.5) node {$1$}; \draw (\x-0.5, -1.5) node {$2$};  }
#1 \end{tikzpicture}}
\ytableausetup{boxsize=1.25em}
\begin{align*}
\tab{
  \draw(1,0) -- (1,-2);
  \draw(2,0) -- (2,-2);
}
& \mapsto
\ytableaushort{{\cdot}{\cdot}{\cdot},{\cdot}{\cdot}{\cdot}}\,,
&
\tab{
  \draw(1,0) -- (1,-1);
  \draw(2,0) -- (2,-2);
}
& \mapsto 
\ytableaushort{{\cdot}{\cdot}{\cdot},{\cdot}{\cdot}{2}}\,,
&
\tab{
  \draw(1,0) -- (1,-2);
  \draw(2,0) -- (2,-1);
}
& \mapsto 
\ytableaushort{{\cdot}{\cdot}{\cdot},{\cdot}{\cdot}{3}}\,,
\\
\tab{
  \draw(1,0) -- (1,-1);
  \draw(2,0) -- (2,-1);
}
& \mapsto 
\ytableaushort{{\cdot}{\cdot}{\cdot},{\cdot}{3}{3}}\,,
&
\tab{
  \draw(2,0) -- (2,-2);
}
&  \mapsto 
\ytableaushort{{\cdot}{\cdot}{1},{\cdot}{\cdot}{1}}\,,
&
\tab{
  \draw(2,0) -- (2,-1);
  \draw(1,-1) -- (1,-2);
}
&  \mapsto 
\ytableaushort{{\cdot}{\cdot}{1},{\cdot}{\cdot}{2}}\,,
\\
\tab{
  \draw(1,0) -- (1,-2);
}
&  \mapsto 
\ytableaushort{{\cdot}{\cdot}{2},{\cdot}{\cdot}{2}}\,,
&
\tab{
  \draw(2,0) -- (2,-1);
}
&  \mapsto 
\ytableaushort{{\cdot}{\cdot}{1},{\cdot}{2}{2}}\,,
&
\tab{
  \draw(1,0) -- (1,-1);
}
&  \mapsto 
\ytableaushort{{\cdot}{\cdot}{2},{\cdot}{2}{2}}\,,
\\
\tab{ }
&  \mapsto 
\ytableaushort{{\cdot}{1}{1},{\cdot}{1}{1}}\,.
\end{align*}
\end{ex}

Next we consider the crystal structure if we consider the reading word using the anchors instead of the buoys.

\begin{ex}
Consider the valued-set tableau from Example~\ref{ex:buoy_reading_word}
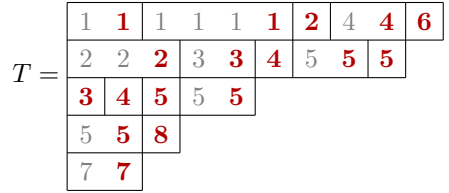
\[
T = 
\begin{tikzpicture}[scale=0.5,baseline=-30]
\foreach \i/\r in {0/10,1/10,2/9,3/5,4/3,5/2}
    \draw (0,-\i) -- (\r,-\i);
\draw (0,0) -- (0,-5);
\draw (1,-2) -- (1,-3);
\draw (2,0) -- (2,-1);
\draw (2,-2) -- (2,-5);
\draw (3,-1) -- (3,-4);
\draw (5,-1) -- (5,-3);
\draw (6,0) -- (6,-2);
\draw (7,0) -- (7,-1);
\draw (8,-1) -- (8,-2);
\draw (9,0) -- (9,-2);
\draw (10,0) -- (10,-1);
\draw (1.5,-0.5) node {$\mathbf{\color{darkred}1}$};
\draw (5.5,-0.5) node {$\mathbf{\color{darkred}1}$};
\foreach \i in {0,2,3,4}
    \draw (\i+0.5,-0.5) node[gray] {$1$};
\draw (6.5,-0.5) node {$\mathbf{\color{darkred}2}$};
\draw (2.5,-1.5) node {$\mathbf{\color{darkred}2}$};
\foreach \i in {0,1}
    \draw (\i+0.5,-1.5) node[gray] {$2$};
\draw (0.5,-2.5) node {$\mathbf{\color{darkred}3}$};
\draw (3.5,-1.5) node[gray] {$3$};
\draw (4.5,-1.5) node {$\mathbf{\color{darkred}3}$};
\draw (7.5,-0.5) node[gray] {$4$};
\draw (8.5,-0.5) node {$\mathbf{\color{darkred}4}$};
\draw (5.5,-1.5) node {$\mathbf{\color{darkred}4}$};
\draw (1.5,-2.5) node {$\mathbf{\color{darkred}4}$};
\draw (6.5,-1.5) node[gray] {$5$};
\draw (7.5,-1.5) node {$\mathbf{\color{darkred}5}$};
\draw (8.5,-1.5) node {$\mathbf{\color{darkred}5}$};
\draw (2.5,-2.5) node {$\mathbf{\color{darkred}5}$};
\draw (3.5,-2.5) node[gray] {$5$};
\draw (4.5,-2.5) node {$\mathbf{\color{darkred}5}$};
\draw (0.5,-3.5) node[gray] {$5$};
\draw (1.5,-3.5) node {$\mathbf{\color{darkred}5}$};
\draw (9.5,-0.5) node {$\mathbf{\color{darkred}6}$};
\draw (0.5,-4.5) node[gray] {$7$};
\draw (1.5,-4.5) node {$\mathbf{\color{darkred}7}$};
\draw (2.5,-3.5) node {$\mathbf{\color{darkred}8}$};
\end{tikzpicture}\,,
\]
where we have written the anchors in bold, and the corresponding reading word is
\[
\rd_{\rightarrow}(T) = 37541852534125546.
\]
We also have
\[
\RSK\bigl( \rd_{\rightarrow}(T) \bigr) = \RSK\bigl( \rd(T) \bigr) =
\ytableaushort{1124456,23555,34,58,7}\,.
\]
\end{ex}

\begin{prop}
\label{prop:same_svt_crystals}
We obtain the same crystal if we use the anchor entries instead of the buoy entries to define the reading word.
\end{prop}

\begin{proof}
To contrast with Lemma~\ref{lemma:crystal_structure_vst}, the $f_i$ operator moves dividers locally as
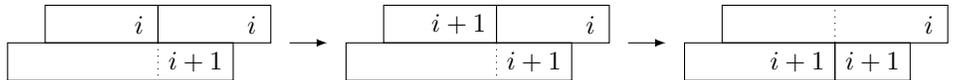
\[
\begin{tikzpicture}[scale=0.5,>=latex,baseline=0]
\draw[-] (0,0) rectangle (3,1);
\draw[-] (3,0) rectangle (6,1);
\draw[-] (-1,0) rectangle (5,-1);
\draw[dotted] (3,0) -- (3,-1);
\draw (4,0) node[anchor=north] {$i+1$};
\draw (2.5,0) node[anchor=south] {$i$};
\draw (5.5,0) node[anchor=south] {$i$};
\draw[->] (6.5,0) -- (7.5,0);
\begin{scope}[xshift=9cm]
\draw[-] (0,0) rectangle (3,1);
\draw[-] (3,0) rectangle (6,1);
\draw[-] (-1,0) rectangle (5,-1);
\draw[dotted] (3,0) -- (3,-1);
\draw (4,0) node[anchor=north] {$i+1$};
\draw (2,0) node[anchor=south] {$i+1$};
\draw (5.5,0) node[anchor=south] {$i$};
\end{scope}
\draw[->] (15.5,0) -- (16.5,0);
\begin{scope}[xshift=18cm]
\draw[-] (0,0) rectangle (6,1);
\draw[-] (-1,0) rectangle (3,-1);
\draw[-] (3,0) rectangle (5,-1);
\draw[dotted] (3,0) -- (3,1);
\draw (2,0) node[anchor=north] {$i+1$};
\draw (5.5,0) node[anchor=south] {$i$};
\draw (4,0) node[anchor=north] {$i+1$};
\end{scope}
\end{tikzpicture}\,.
\]
This also preserves the reading word and yields the same result.
Additionally note that Proposition~\ref{prop:vst_hw_condition} and the bijection constructed in proving Corollary~\ref{cor:vst_hw_counting} holds in either reading word.
Since all highest weight words of weight $\mu$ generate the same crystal $B(\mu)$, the claim follows.
\end{proof}

Proposition~\ref{prop:same_svt_crystals} holds if we instead define the reading word based on selecting any entry in any group as long as the crystal operators move the selected entry within each column (when moving dividers).

\begin{remark}
In~\cite[Remark~9]{Galashin17}, it was essentially noticed that the crystal structure on reverse plane partitions by taking the topmost entry in each column is the same as taking the bottommost. This is analogous to Proposition~\ref{prop:same_svt_crystals}.
\end{remark}

\subsection{Inflation map}

Define a \defn{column semistandard flagged tableau} to be a conjugate semistandard whose $i$-th column is strictly flagged by $i$. Let $\mcF^{cs}_{\lambda/\mu}$ denote the set of column semistandard flagged tableaux of shape $\lambda/\mu$.

Next, we construct an explicit crystal isomorphism
\[
\iota \colon \vst^n(\lambda) \to \bigsqcup_{\mu \subseteq \lambda} B(\mu) \times \mcF^{cs}_{\lambda/\mu},
\]
where the crystal structure on the codomain is given by $f_i(b \times F) = (f_i b) \times F$ for all $b \times F \in B(\mu) \times \mcF^{cs}_{\lambda/\mu}$ for any fixed $\mu$.
We call the map $\iota$ \defn{inflation}.
For any $T \in \vst^n(\lambda)$, we define $\iota(T)$ recursively starting with $b_{\lambda_1+1} \times F_{\lambda_1+1} = \emptyset \times \emptyset$. Suppose we are at step $i$ with the current state being $b_i \times F_i$, and let $C_j$ denote the $j$-th column of $T$ consisting only of the \emph{anchor} entries. Construct
\[
b_{i-1} := \rd(C_{i-1}) \xleftarrow{\RSK} \rd(C_i) \xleftarrow{\RSK} \cdots \xleftarrow{\RSK} \rd(C_{\lambda_1}).
\]
Construct $F_{i-1}$ by starting first with $F_i$ of shape $\mu_i$ but shifting the elements to the right one step and increasing them by $1$, which partially fills in the shape $\lambda_{\geq i-1} / \mu_{i-1}$, where $\lambda_{\geq i-1}$ is from Equation~\eqref{eq:rightmost_column_shape} and $\mu_{i-1}$ is the shape of $b_{i-1}$.
Then set all of the unfilled boxes of $F_{i-1}$ to $1$. Thus, we constructed the $(i-1)$-th step $b_{i-1} \times F_{i-1}$. Repeating this for every column, the final result is $\iota(T) = b_1 \times F_1$.

\begin{ex}
\label{ex:inflation}
Let $T$ be the valued-set tableau from Example~\ref{ex:buoy_reading_word}. Applying inflation to $T$ we obtain
{\allowdisplaybreaks
\begin{align*}
\ytableausetup{boxsize=1.5em}
b_{11} \times F_{11} & = \emptyset & \times \quad & \emptyset,  
\\ b_{10} \times F_{10} & = \ytableaushort{6} & \times \quad & \ytableaushort{{\cdot}}\,,  
\\ b_9 \times F_9 & = \ytableaushort{46,5} & \times \quad & \ytableaushort{{\cdot}{\cdot},{\cdot}}\,,  
\\ b_8 \times F_8 & = \ytableaushort{456,5} & \times \quad & \ytableaushort{{\cdot}{\cdot}{\cdot},{\cdot}{\ml{1}}}\,,  
\\ b_7 \times F_7 & = \ytableaushort{2456,5} & \times \quad & \ytableaushort{{\cdot}{\cdot}{\cdot}{\cdot},{\cdot}{\ml{1}}2}\,,  
\\ b_6 \times F_6 & = \ytableaushort{12456,45} & \times \quad & \ytableaushort{{\cdot}{\cdot}{\cdot}{\cdot}{\cdot},{\cdot}{\cdot}23}\,,  
\\ b_5 \times F_5 & = \ytableaushort{12456,345,5} & \times \quad & \ytableaushort{{\cdot}{\cdot}{\cdot}{\cdot}{\cdot}{\ml{1}},{\cdot}{\cdot}{\cdot}45,{\cdot}}\,,  
\\ b_4 \times F_4 & = \ytableaushort{12456,345,5} & \times \quad & \ytableaushort{{\cdot}{\cdot}{\cdot}{\cdot}{\cdot}{\ml{1}}2,{\cdot}{\cdot}{\cdot}{\ml{1}}56,{\cdot}{\ml{1}}}\,,  
\\ b_3 \times F_3 & = \ytableaushort{124456,235,55,8} & \times \quad & \ytableaushort{{\cdot}{\cdot}{\cdot}{\cdot}{\cdot}{\cdot}23,{\cdot}{\cdot}{\cdot}{\ml{1}}267,{\cdot}{\cdot}2,{\cdot}}\,,  
\\ b_2 \times F_2 & = \ytableaushort{1124456,2355,45,58,7} & \times \quad & \ytableaushort{{\cdot}{\cdot}{\cdot}{\cdot}{\cdot}{\cdot}{\cdot}34,{\cdot}{\cdot}{\cdot}{\cdot}2378,{\cdot}{\cdot}{\ml{1}}3,{\cdot}{\cdot},{\cdot}}\,,  
\\ b_1 \times F_1 & = \ytableaushort{1124456,23555,34,58,7} & \times \quad & \ytableaushort{{\cdot}{\cdot}{\cdot}{\cdot}{\cdot}{\cdot}{\cdot}{\ml{1}}45,{\cdot}{\cdot}{\cdot}{\cdot}{\cdot}3489,{\cdot}{\cdot}{\ml{1}}24,{\cdot}{\cdot}{\ml{1}},{\cdot}{\ml{1}}}\,,
\end{align*}
}
resulting in $\iota(T) = b_1 \times F_1$.
\end{ex}

\begin{thm}
We have
\[
\vst^n(\lambda) \iso \bigoplus_{\mu \subseteq \lambda} B(\mu)^{\oplus \abs{\mcF^{cs}_{\lambda/\mu}}},
\]
where the isomorphism is given by the inflation map $\iota$, and so $V_{\lambda}^{\mu} = \abs{\mcF^{cs}_{\lambda/\mu}}$. Moreover, we have
\[
\dwG_{\lambda}(\xx; \alpha) = \sum_{\mu \subseteq \lambda} \alpha^{\abs{\lambda} - \abs{\mu}} \abs{\mcF^{cs}_{\lambda/\mu}} s_{\mu}(\xx).
\]
\end{thm}

\begin{proof}
This is similar to the proof as for~\cite[Thm.~9.8]{LamPyl07} except for conjugating by the Lusztig involution and working column-by-column as in Theorem~\ref{thm:flagged_decomposition}. Indeed, the $i$-th step of $\iota$ adds a vertical strip to $b_i$ by well-known properties of RSK, and thus these boxes added to form $b_{i-1}$ are precisely the boxes that are not filled with a $1$ in $F_{i-1}$. Therefore, the inverse map (deflation) $\iota^{-1}$ is given recursively by performing inverse RSK (with conjugating by the Lusztig involution) on the boundary boxes of $b_i$ that do not immediately have a $1$ to the right in $F_i$. It is clear that this is the inverse procedure (the $i$-th column of the resulting valued-set tableaux is the leftmost column of step $i$) and the claim follows.
\end{proof}

\begin{remark}
The inflation map for valued-set tableaux is analogous to the bijection from~\cite[Thm.~9.8]{LamPyl07} for reverse plane partitions, which is also a crystal isormorphism using the crystal structure from~\cite{Galashin17}. Furthermore, as noted in the proof, this analogy is the same as that for the uncrowding multiset-valued tableaux to uncrowding set-values tableaux.
\end{remark}

\begin{ex}
Consider $b_1 \times F_1$ from Example~\ref{ex:inflation}. We will take the Lusztig involution in $\fsl_9$. Thus we have
\[
b_1^* = \ytableaushort{1244446,3457{\ml{8}},55,68,7} \; = \; \ytableaushort{1244448,3457,55,68,7} \xleftarrow{\RSK} 6,
\]
where we performed inverse row bumping on the bold letter. Thus we have an anchor entry of $3 = 6^*$ in the first column of the resulting valued-set tableau of $\iota^{-1}(b_1 \times F_1)$ (the rest of the entries of the column are non-anchor entries in their groups). For the next step, we have
\[
b_2^* = \ytableaushort{124444{\ml{8}},345{\ml{7}},55,6{\ml{8}},{\ml{7}}} \; = \; \ytableaushort{144447,356,58,7} \xleftarrow{\RSK} 8542,
\]
and $(8542)^* = 7541$. We have now reconstructed the first two columns of $T$.
\end{ex}

\bibliographystyle{alpha}
\bibliography{crystals}{}
\end{document}